\documentclass[11pt]{article}

\usepackage{graphicx,epsfig,color}
\usepackage{amsmath,amssymb,amsthm}
\usepackage{bbm}
\usepackage{mathtools}
\usepackage{xr-hyper}
\newcommand{\R}{ \mathbb{R}}          
\newcommand{\N}{ \mathbb{N}}          
\def\V#1{{\boldsymbol{#1}}}           
\def\M#1{{\bf{#1}}}                   
\newcommand{\Tr}{\mathsf{T}}          
\newcommand{\Fourier}{\Spc F}          

\def\Spc#1{{\mathcal{#1}}}            
\def\Op#1{{\mathrm{#1}}}              
\def\ee{\mathrm{e}}                   
\def\jj{\mathrm{i}}                   
\newcommand{\dint}{{\rm d}}           
\def\Indic{\mathbbm{1}} 

\newcommand{\embedD}{\overset{d}{\hookrightarrow}} 
\newcommand{\embedC}{\overset{}{\hookrightarrow}} 
\newcommand{\eqdef}{:=}
\newcommand{\Lop}{\Op L}
\newcommand{\Dop}{\Op D}
\newcommand{\Identity}{\mathrm{Id}}
\newcommand{\bw}{\V \omega}

\theoremstyle{plain}
\newtheorem{theorem}{Theorem}[section]
\newtheorem{lemma}[theorem]{Lemma}

\newtheorem{proposition}[theorem]{Proposition}
\theoremstyle{definition}
\newtheorem{definition}[theorem]{Definition}

\theoremstyle{remark}
\newtheorem{remark}[theorem]{Remark}

\usepackage[hidelinks]{hyperref}
\begin{document}

\title{Splines and Operators:\\ Part II---Resolution of Inverse Problems}
\author{Michael Unser\thanks{Biomedical Imaging Group, \'Ecole polytechnique
f\'ed\'erale de Lausanne (EPFL), CH-1015 Lausanne, Switzerland
(michael.unser@epfl.ch).}}
\date{}
\maketitle

\begin{abstract}
We present a unified treatment of the data-fitting problems with quadratic (Hilbertian)
regularization that arise in approximation theory, machine learning, engineering, and
statistics. The generic problem is an optimization in the continuum: given $M$ linear
measurements of an unknown function, one minimizes a convex data-fidelity term that is
penalized by $\|\Lop f\|^2_{L_2}$, where $\Lop$ is a differential-like (spline-admissible)
operator. The natural domain of the minimization is the native space $\Spc H_\Lop$ of
$\Lop$---the largest Hilbert space over which the regularizer remains meaningful---for which we
provide a concise and directly usable characterization. Central to our formulation is a
measurement-adapted (conditioned) basis that decouples the null-space and spline components of
the solution. This basis serves in three places. First, it tightens the formulation of our representer
theorem, which confines the unique solution to an $M$-dimensional space without any
orthogonality constraint on the coefficients. Second, it streamlines the implementation, as the exact
discretization of the problem becomes an unconstrained convex program with closed-form
solutions in the least-squares case, which covers ridge regression, Tikhonov regularization,
and smoothing splines. Third, it makes the bridge with the minimum-mean-square-error estimation of
generalized Gaussian processes, which is the topic of a companion paper.
\end{abstract}
\pagebreak
\tableofcontents
\pagebreak
\bigskip

\section{Introduction}
\label{Sec:Introduction}

The recovery of a function---or, more generally, of a distribution---from a finite number of
linear measurements is the prototypical infinite-dimensional inverse problem. Since the task is intrinsically
ill-posed, the standard remedy is regularization: among all candidates that are consistent with
the data, one selects the one that minimizes a regularization functional. In this paper, the regularizer is
$\|\Lop f\|^2_{L_2}$, where $\Lop$ is a differential-like operator, and the corresponding
solutions are {\em generalized splines}. Specifically, let $\M y \in \R^M$ be the data that are
collected by the linear measurement functionals $\V \nu=(\nu_1,\dots,\nu_M)$. The generalized spline
fit of these data is then the solution of the constrained optimization problem
\begin{align}
\label{Eq:GsplineInt}
\min_{f \in \Spc H_\Op L} \|\Lop f\|_{L_2} \mbox{ s.t. } \V \nu(f)=\M y,
\end{align}
where $\Spc H_\Lop \subset \Spc S'$ is the so-called {\em native space} of $\Lop$; that
is, the largest space over which the regularizer remains meaningful. The most interesting (and also
more challenging) forms of the problem are the ones where $\Lop$ has a nontrivial null space
$\Spc N_{\Lop}=\Spc N_{\V p}$ that is spanned by a finite-dimensional basis
$\V p=(p_1,\dots,p_{N_0})$ with $p_n \in \Spc S'$ and $0 < N_0 <M$, the prototypical example
being the polynomials of degree less than $N_0$ for $\Lop=\Dop^{N_0}$.

Problem \eqref{Eq:GsplineInt} and its noise-tolerant variants lie at the intersection of several classical threads that have developed largely in parallel. The first is the variational theory of splines, which originates with Schoenberg \cite{Schoenberg1964}, was formalized by de Boor and Lynch \cite{deBoor1966} with the help of reproducing-kernel Hilbert spaces (RKHS) \cite{Aronszajn1950}, and was subsequently abstracted to general Hilbert spaces by the French school \cite{Anselone1968,Atteia1992}. This theory was extended to generalized measurement functionals with the Lg-splines of Jerome and Schumaker \cite{Jerome1969}, to the multidimensional setting with the thin-plate splines of Duchon \cite{Duchon1977}, and to noisy data with the smoothing splines of Wahba \cite{Wahba1990}; see \cite{Prenter:1975,Schumaker2007} for textbook accounts.
The second thread is statistical. The representer theorems of machine learning, from Kimeldorf and Wahba \cite{Kimeldorf1971} to their modern kernel-based \cite{Scholkopf2001,Hofmann2008,Berlinet2004} and measurement-adapted \cite{Gupta2018,Badoual2018,Unser_2020} variants, establish that the minimizers of regularized fitting problems are finite linear combinations of adapted basis functions. Their simplest finite-dimensional instances are ridge regression \cite{Hoerl1970} and Tikhonov regularization \cite{Tikhonov1977}, which are workhorses of statistics and inverse-problem theory. More recently, representer theorems have been established beyond the Hilbertian setting, both for sparsity-promoting regularization in Banach spaces \cite{Boyer2018,Unser2022,Bredies2018} and for the function spaces associated with neural networks \cite{Parhi2021b,Bartolucci2021}. The present paper remains within the quadratic framework, where stronger and fully explicit statements are available.
The third thread is approximation-theoretic. In the theory of radial basis functions, the native space is classically constructed as the completion of a span of kernel translates \cite{Schaback1999,Schaback2000,Wendland2005}, with conditionally positive-definite kernels used to accommodate polynomial null spaces \cite{Micchelli1986}. The closest point of contact with our operator-based approach is the Green-function construction of reproducing kernels for generalized Sobolev spaces \cite{Fasshauer2011}.

Our goal in this paper is threefold. First, we seek a unified treatment of this class of problems, which all
share a quadratic (Hilbertian) regularization: one theory, one solution structure, and one
algorithm. Second, we wish to identify the broadest possible class of infinite-dimensional
optimization problems whose solutions are confined to a finite-dimensional linear subspace.
This requires the determination of the acceptable regularization operators $\Lop$, of the space
over which the minimization is well-posed, of the admissible measurement functionals, and of
the losses that preserve the structure of the solution. Third, we aim at a common numerical
resolution pipeline: an exact discretization of the continuous-domain problem, with closed-form
solutions whenever the data-fidelity term is quadratic.

This paper is the second of a series. Part I \cite{Unser2026a} characterizes the operators $\Lop$ for which
the program is feasible ($\V p$-admissibility) and constructs the native space $\Spc H_\Lop$
abstractly, by completion of a space of test functions. The setting is that of a nuclear Gelfand
pair $(\Spc S,\Spc S')$, which yields the sandwich
$\Spc S \embedD \Spc H_\Lop \embedD \Spc S'$: the density of $\Spc S$ ensures universality ($\Spc
H_\Lop$ is rich enough to approximate any signal as closely as desired), while the continuous
embedding in $\Spc S'$ ensures well-posedness and provides the calculus of adjoints and Riesz
maps. The generalization is of the same nature as the passage from classical functions to
distributions; it covers all $\V p$-admissible operators at once, including fractional and
multidimensional ones. In the present Part II, we convert this abstract machinery into a
resolution theory for linear inverse problems. Let us state our contributions with the proper
caveat: the first two are consolidations and extensions---at the announced level of
generality---of results whose classical instances are well known, while the third one is, to
the best of our knowledge, genuinely new.
\begin{enumerate}
\item {\em A ready-to-use description of the native space}
(Theorem \ref{Theo:NativeSpace}): a self-contained description of $\Spc H_\Lop$---inner product,
direct-sum decomposition, embeddings, Riesz maps, and explicit representation of its
members in terms of $\Lop$ only. The generality of this structure is inherited from the
backbone nuclear space $\Spc S$: this abstract setting covers all $\V p$-admissible
operators, including the fractional and multidimensional ones. The description extends to
non-universal biorthogonal systems, which is precisely what the classical spline examples
require.
\item {\em An extension of the classic representer theorems of machine learning to their
maximal level of generality}
(Theorem \ref{Theo:GeneralRepSemiHilbert}): for a generic convex loss and arbitrary continuous
linear measurements over $\Spc H_\Lop$, the unique solution lies in an $M$-dimensional space
that is fixed by $(\Lop,\V \nu)$ alone. Interpolation, regression, Tikhonov regularization,
kernel estimation, and smoothing splines all arise as special cases, and each hypothesis of
the theorem answers one facet of the second goal.
\item {\em The measurement-adapted (conditioned) basis}
(Lemma \ref{Prop:DirectSum}): a biorthogonal system that is matched to the measurement
functionals and that separates the null-space and spline components of the solution once and
for all. This basis helps in three distinct places. It tightens the formulation of the
representer theorem, whose parametrization \eqref{Eq:f0lin2} is inherently stable---the basis
functions lie in $\Spc H_\Lop$---and free of the traditional orthogonality constraints on the
coefficients. It simplifies the implementation, because the exact discretization of
Section \ref{Sec:Discretization} then becomes an unconstrained convex program, with closed-form
and fully decoupled solutions in the least-squares case. 
\end{enumerate}

The paper is organized as follows. In Section \ref{Sec:Preliminaries}, we recall the primary
definitions and results of Part I: Hilbert--Gelfand spaces, biorthogonal systems, and
spline-admissible operators. In Section \ref{Sec:NativeSpace}, we establish the announced
characterization of the native space, together with a description of the range of the
conditional inverse (Proposition \ref{Prop:RangeA}). In Section \ref{Sec:InverseProblems}, we
develop the resolution theory: we establish the representer theorem for generalized variational
splines (Section \ref{Sec:VarSpline}), make the link with RKHS and conventional splines
(Section \ref{Sec:RKHS}), and derive the exact discretization together with its closed-form
solutions (Section \ref{Sec:Discretization}). 

Let us close this introduction with a pointer to the stochastic side of the story. Kimeldorf and
Wahba \cite{Kimeldorf1970} uncovered a correspondence between spline fitting and the Bayesian
estimation of Gaussian processes, which was further developed within the RKHS framework of
statistical estimation theory \cite{Parzen1961,Kailath1970,Weinert1974}; for modern accounts of
this deterministic--stochastic correspondence, we refer to \cite{Scheuerer2013,Kanagawa2018}. A companion paper
\cite{Unser2026c} extends this correspondence to the present level of generality: the
generalized splines derived here will be identified as minimum-mean-square-error estimators of
generalized Gaussian processes \cite{Unser2014book}, with the conditioned parametrization of
Theorem \ref{Theo:GeneralRepSemiHilbert} being precisely the ingredient that keeps the
stochastic counterpart well-defined.

\section{Preliminaries: Gelfand Pairs and Admissibility}
\label{Sec:Preliminaries}

Here, we briefly recall the primary definitions and theoretical outcomes of Part I \cite{Unser2026a}.
\subsection{Hilbert--Gelfand Spaces}
\begin{definition}[Nuclear Gelfand pair]
\label{Def:NuclearGelfandPair}
The couple $(\Spc S,\Spc S')$ that is formed by a topological vector space $\Spc S$ and
its topological dual $\Spc S'$ is called a \emph{nuclear Gelfand pair} if
\begin{enumerate}
\item $\Spc S$ is nuclear and Montel
\cite{Grothendieck1955,Pietsch1972,Treves2006};
\item there exists a continuous and dense embedding $\iota:\Spc S\to\Spc S'$ that is
canonical, in the sense that
$\langle\iota\varphi,\psi\rangle=\langle\iota\psi,\varphi\rangle$ for all
$\varphi,\psi\in\Spc S$.
\end{enumerate}
Unless stated otherwise, $\Spc S'$ is endowed with its weak-$\ast$ topology. 
\end{definition}

To make matters concrete, we may pick the pair $\big(\Spc S(\R^d),\Spc S'(\R^d)\big)$ where $\Spc S(\R^d)$ is the Schwartz space of smooth, rapidly decaying test functions on $\R^d$, and $\Spc S'(\R^d)$ is the space of tempered distributions. 
The corresponding canonical embedding
$\iota\colon\Spc S(\R^d)\hookrightarrow\Spc S'(\R^d)$ is specified as
\begin{align}
\label{Eq:CanonicalEmbedding}
\iota\varphi: \psi \mapsto \langle \iota\varphi,\psi\rangle \eqdef \int_{\R^d} \varphi(\V x)\,\psi(\V x)\dint\V x,%
\qquad \psi\in\Spc S(\R^d).
\end{align}
The latter enables us to identify a test function
$\varphi\in\Spc S(\R^d)$ with the tempered distribution
$\iota\varphi\in\Spc S'(\R^d)$; sometimes, we shall simplify the notation to $\varphi\in\Spc S'(\R^d)$ (respectively, $\varphi: \psi\mapsto \langle \varphi, \psi \rangle$) when the interpretation is unambiguous. 

A vector space $\Spc H$ is said to be a \emph{$\Spc S$-Hilbert--Gelfand space} if it contains $\Spc S$ and if it is continuously embedded in $\Spc S'$.
Since $\Spc S$ is dense in $\Spc S'$, the Hilbert-Gelfand property can be expressed as
\begin{align}
\label{Eq:HilbertGelfand}
\Spc S\embedD \Spc H\embedD\Spc S'
\end{align}
where the symbol ``$\embedD$'' denotes a continuous and dense embedding. The three properties of Hilbert--Gelfand spaces that are fundamental for our purpose are: 
\begin{enumerate}
\item The preservation of the property through duality; i.e., $\Spc S\embedD \Spc H'\embedD\Spc S'$, and the compatibility of the duality product for $(\Spc H,\Spc H')$ with that of $(\Spc S, \Spc S')$. Moreover, since $\Spc S\embedC \Spc H$, we have that
\begin{align}
\label{Eq:DualityBound}
|\langle \nu, \varphi \rangle|\le \|\nu\|_{\Spc H'} \|\varphi\|_{\Spc H}
\end{align}
for any $\nu \in \Spc H'\subset \Spc S'$ and $\varphi \in \Spc S$, and, by continuous extension, for $\varphi \in \Spc H$.
\item The possibility of a $\Spc S$-based construction by completion as
\begin{align}
\label{Eq:HilbertCompletion}
\Spc H=\overline{(\Spc S,\|\cdot\|_{\Spc H})},
\end{align}
where $\|\cdot\|_{\Spc H}$ is the Hilbertian norm for $\Spc H$.
\item The retrieval of the inner product as\begin{align}
\label{Eq:RieszPairing}
\forall f,g \in \Spc H:\quad \langle f,g \rangle_{\Spc H}=\langle \Op J_{\Spc H}f, g\rangle=\langle f,\Op J_{\Spc H}g\rangle, 
\end{align}
where $\Op J_{\Spc H}: \Spc H \to \Spc H'$ is the (unique) Riesz map that isometrically maps $\Spc H$ into its topological dual.
\end{enumerate}

The baseline example of a Hilbert--Gelfand space is the pivot space
\begin{align}
\label{Eq:L2}
L_2=\overline{(\Spc S,\|\cdot\|_{L_2})},\quad \|\varphi\|^2_{L_2}=\langle \varphi,\varphi \rangle:=\langle \iota \varphi,\varphi \rangle
\end{align}
whose existence is guaranteed by the second condition in Definition \ref{Def:NuclearGelfandPair}. For $\Spc S=\Spc S(\R^d)$, Eq.~\eqref{Eq:L2} yields
$L_2=L_2(\R^d)$, which coincides with Lebesgue's space of finite energy functions. The main point for us is that there is a ``canonical'' $L_2$ space associated with every nuclear Gelfand pair $(\Spc S, \Spc S')$.
\subsection{Vector-Functionals and Biorthogonal Systems}
\label{Sec:VecFunctional}
We use $\V \nu=(\nu_1,\dots,\nu_{M})$ to denote a collection of linear functionals with $\nu_m \in \Spc S'$ (or $\Spc H'$), $m=1, \dots, M$.
Under the implicit assumption that the $\nu_m$ are linearly independent, we then define: 
\begin{enumerate}
\item the linear subspace
$\Spc N_{\V \nu}={\rm span}\{\nu_m\}_{m=1}^{M} \subset \Spc S'$;
\item the linear vector-functional
\begin{align}
\V \nu: \varphi \mapsto \V \nu(\varphi)=\big (\langle \nu_1, \varphi\rangle,\dots, \langle \nu_{M}, \varphi\rangle\big):  \Spc S \to \R^{M};
\end{align}
\item the symmetric, finite-rank operator $\Op R_{\V \nu}: \Spc S \to \Spc N_{\V \nu}\subset \Spc S'$ (or  $\Op R_{\V \nu}: \Spc H \to \Spc N_{\V \nu}\subset \Spc H'$) with
\begin{align}
\Op R_\V \nu: \varphi \mapsto \Op R_\V \nu\{\varphi\}=\sum_{m=1}^M\nu_m \langle \nu_m, \varphi\rangle.%
\end{align}
The same formula is read in the reverse direction, as $\Op R_{\V \nu}: \Spc S' \to \Spc S$ (or
$\Op R_{\V \nu}: \Spc H' \to \Spc N_{\V \nu}\subset\Spc H$), whenever the members of the family
are regular; i.e., $\V \nu \subset \Spc S$ (respectively, $\V \nu \subset \Spc H$).
\end{enumerate}

We also briefly recall how the space $\Spc N_\V p$ induced by $\V p=(p_1,\cdots,p_{N_0}) \subset \Spc S'$ (e.g., the attainable null space of the operator $\Op L$) can be turned into a Hilbert space, the non-standard part being that the
$p_n$ (which are typically Taylor monomials) are not included in $L_2$, so that they cannot be orthogonalized. Our solution  
relies on some biorthogonal basis $\V \phi=(\phi_1,\dots,\phi_{N_0})$ (see Part I, Lemma~\ref{I-Lem:Biorthogonal}).
The latter is generally not unique, unless
one forces the $\phi_n$ to be part of $\Spc N_{\V p}$, which is only possible if $\Spc N_{\V p}\subset L_2$.

\begin{proposition}[Finite-dimensional Hilbert subspaces of $\Spc S'$]
\label{Prop:FiniteDimHilbert}
Let $\V p=(p_n)$ be a basis of $\Spc N_{\V p}={\rm span}\{p_n\}_{n=1}^{N_0}\subset \Spc S'$. Then, there exists
 a vector of functionals $\V \phi=(\phi_n)$ with $\phi_1,\dots,\phi_{N_0}\in \Spc S$ such that
 $\langle p_n, \phi_m\rangle = \delta_{m,n}$ (biorthogonality).
 
The resulting biorthogonal system $(\V p, \V \phi)$ is called {\em universal} and it allows us to identify the dual pair of 
Hilbert spaces  $\Spc H_0=(\Spc N_{\V p},|\cdot|_{\Op R_\V \phi}) \subset \Spc S'$ and
$\Spc H'_0=(\Spc N_{\V \phi},|\cdot|_{\Op R_\V p}) \subset \Spc S$ equipped with the (semi-)inner
products
\begin{align}
\langle f_1,f_2\rangle_{\Op R_\V \phi} = \sum_{n=1}^{N_0} \langle \phi_n,f_1 \rangle\langle \phi_n,f_2 \rangle\\ 
\langle \varphi_1,\varphi_2\rangle_{\Op R_\V p} = \sum_{n=1}^{N_0} \langle p_n,\varphi_1 \rangle \langle p_n,\varphi_2 \rangle,
\end{align}
which are well-defined for any $f_1,f_2 \in \Spc S'$ and $\varphi_1,\varphi_2 \in \Spc S$.
The same construction applies with any dual pair $(\Spc H,\Spc H')$ of Hilbert spaces such that
$\Spc S \embedC \Spc H \embedC \Spc S'$ in place of $(\Spc S,\Spc S')$, which enables the use of
{\em non-universal} biorthogonal systems.
\end{proposition}

Given a dual Hilbert--Gelfand pair of spaces $(\Spc H, \Spc H')$ and a choice of biorthogonal system $(\V p, \V \phi)$ with $\V p \subset \Spc H'$ and $\V \phi \subset \Spc H$, we can then use the formalism to specify the pair of oblique projectors
\begin{align}
\label{Eq:Projectors}
\Op P_{\V \phi}&\eqdef\Identity-\Op R_{\V \phi}\Op R_{\V p}: \Spc H \to \Spc H_{\V p}\eqdef\{\varphi \in \Spc H: \V p(\varphi)=\V 0\}\\
\Op P^\ast_{\V \phi}&\eqdef\Identity-\Op R_{\V p}\Op R_{\V \phi}: \Spc H' \to \Spc H'_{\V \phi}\eqdef\{f \in \Spc H': \V \phi(f)=\V 0\},
\end{align}
which are the adjoint of one another. Their complements $\Identity-\Op P_{\V \phi}=\Op R_{\V \phi}\Op R_{\V p}: \Spc H \to \Spc N_{\V \phi}$ and $\Identity-\Op P^\ast_{\V \phi}=\Op R_{\V p}\Op R_{\V \phi}: \Spc H' \to \Spc N_{\V p}$ are the projectors on the finite-dimensional components, which gives the two complementary splittings $\Spc H=\Spc H_{\V p}\oplus \Spc N_{\V \phi}$ and $\Spc H'=\Spc H'_{\V \phi}\oplus \Spc N_{\V p}$. This is the notation of Part I, where these projectors are constructed for $(\Spc H,\Spc H')=(\Spc S,\Spc S')$, with $\Spc S_{\V p}$ and $\Spc S'_{\V \phi}$ as ranges.
\subsection{Spline-Admissible Operators}
Throughout, $\Spc S_{\V p}=\{\varphi\in\Spc S:\V p(\varphi)=\V 0\}$ denotes the annihilator of
$\Spc N_{\V p}$ in $\Spc S$, which is the space $\Spc H_{\V p}$ of \eqref{Eq:Projectors} for
$\Spc H=\Spc S$.

\begin{definition}[Conditional positivity]
\label{Def:CondPositiveDef}
Let $(\Spc S, \Spc S')$ be a nuclear Gelfand pair and $\Spc N_{\V p}$ a
finite-dimensional subspace of $\Spc S'$ spanned by $\V p=(p_1,\dots,p_{N_0})$. Then, the continuous operator
$\Op A: \Spc S \to \Spc S'$ is said to be {\em $\V p$-conditionally positive-definite} if
\begin{align}
\label{Eq:CondPositiveDefinite}
\langle \Op A \varphi,\varphi\rangle > 0
\mbox{ for all }\varphi \in \Spc S_{\V p} \backslash\{0\}.
\end{align}
Likewise, $\Op B: \Spc S \to \Spc S'$ is said to be {\em positive-definite} if
\begin{align}
\label{Eq:PositiveDefinite}
\langle \Op B \varphi,\varphi\rangle > 0
\mbox{ for all }\varphi \in \Spc S \backslash\{0\}.
\end{align}
\end{definition}
The latter property is obviously stricter: the (unconditional) positive definiteness of an operator implies that it is self-adjoint (i.e., symmetric) and $\V p$-conditionally positive for any $\Spc N_\V p \subset \Spc S'$. 

For any pair of operators $(\Op A,\Op B)$ that meet the continuity and positivity conditions of Definition \ref{Def:CondPositiveDef}, it is then possible, by completion, to form the two
Hilbert spaces
\begin{align}
\label{Eq:HA} 
\Spc H_{\Op A}\eqdef\overline{(\Spc S_{\V p},|\cdot|_{\Op A})}\\%
\label{Eq:HB} 
\dot{\Spc H}_{\Op B}\eqdef\overline{(\Spc S,|\cdot|_{\Op B})},
\end{align}
taken in the abstract sense (equivalence classes of Cauchy sequences). Whether or not these spaces are genuine
subspaces of $\Spc S'$---as in the Hilbert--Gelfand configuration \eqref{Eq:HilbertGelfand}---will depend on the fundamental notion
of {\em attainable null space}.

\begin{definition}[Attainable null space]
\label{Def:AttainableNull}
Let $q$ be a continuous seminorm on $\Spc S$. Its \emph{attainable null space} is
\begin{align*}
\Spc N_q\eqdef\big\{p\in\Spc S':\ p=\textstyle\lim_{k}\chi_k\ \text{in}\ \Spc S'
\ \text{for some}\ (\chi_k)\subset\Spc S\ \text{with}\ q(\chi_k)\to0\big\}.
\end{align*}
For the seminorm induced by a positive-semi-definite operator $\Op B:\Spc S\to\Spc S'$, we write
$\Spc N_{\Op B}\eqdef\Spc N_{|\cdot|_{\Op B}}$.
\end{definition}
The above definition can also be extended to the conditional scenario of \eqref{Eq:HA} by
replacing $\Spc S$ by $\Spc S_\V p$. The main point is that the abstract Hilbert space
specified by \eqref{Eq:HB} \big(resp., by \eqref{Eq:HA}\big) is genuinely embedded in $\Spc S'$
if and only if its attainable null space is trivial; otherwise, the completion consists of
equivalence classes, which is what the dotted notation $\dot{\Spc H}_{\Op B}$ records. In the
configuration below, referred to as a $\V p$-conditional Riesz pair,
$\Spc N_{\Op B}=\Spc N_{\V p}$ is finite-dimensional, so that the primary space
$\Spc H_{\Op A}$ is a genuine Hilbert subspace of $\Spc S'$, while the \emph{homogeneous space}
$\dot{\Spc H}_{\Op B}\cong(\Spc H_\Op A)'$ has a concrete description as a subspace of the
quotient $\Spc S'/\Spc N_{\Op B}$, with $\pi:\Spc S'\to\Spc S'/\Spc N_{\Op B}$ the
corresponding quotient map. 

\begin{definition}[$\V p$-conditional Riesz pair]
\label{Def:Intrinsic}
Let $\Op A,\Op B:\Spc S\to\Spc S'$ be continuous, with $\Op A$ $\V p$-conditionally positive-definite and $\Op B$ positive-definite, so that the Hilbert spaces $\Spc H_{\Op A}$ and
$\dot{\Spc H}_{\Op B}$ of \eqref{Eq:HA}  and \eqref{Eq:HB}  are well defined. 
Then, the pair $(\Op A,\Op B)$ is a \emph{$\V p$-conditional Riesz pair} for $(\Spc S,\Spc S')$ if
\begin{enumerate}
\item[(i)] {\em Null-space matching}: $\Spc N_{\Op B}=\Spc N_{\V p}$; and
\item[(ii)] {\em Hilbert duality}: the map $\pi\,\Op A$ and the operator $\Op B$ extend to mutually inverse isometric
isomorphisms
\begin{align*}
\pi\,\Op A:\ \Spc H_{\Op A}\ \longrightarrow\ \dot{\Spc H}_{\Op B},
\qquad
\Op B:\ \dot{\Spc H}_{\Op B}\ \longrightarrow\ \Spc H_{\Op A}.
\end{align*}
\end{enumerate}
\end{definition}

\begin{definition}[Spline-admissible operator]
\label{Def:SplineAdmissible}
Let $(\Spc S,\Spc S')$ be a nuclear Gelfand pair. A continuous linear operator $\Op L: \Spc S \to L_2$ with {\em normal operator} $\Op B=\Op L^\ast \circ \Op L: \Spc S \to L_2 \to \Spc S'$
is said to be $\V p$-admissible if
\begin{enumerate}
\item There exists some operator $\Op A: \Spc S \to \Spc S'$ (the {\em conditional inverse} of the normal operator $\Op B$) such that $(\Op A, \Op B)$ is a $\V p$-conditional Riesz pair. 
\item The (homogeneous) Hilbert space $\dot{\Spc H}_\Op L=\overline{(\Spc S,\|\Op L \cdot\|_{L_2})}$ is isometrically isomorphic to $L_2$.
\end{enumerate}
\end{definition}
Since the first condition in Definition \ref{Def:SplineAdmissible} ensures that $\dot{\Spc H}_\Op L$ is isometrically isomorphic
to the (predual) Hilbert space 
\begin{align}
\label{Eq:NormA}
\Spc H_\Op A=\overline{(\Spc S_\V p,|\cdot|_{\Op A})}, \quad |\psi|_{\Op A}:=\sqrt{\langle \Op A \psi,\psi\rangle}\mbox{ with } \psi \in \Spc S_\V p,
\end{align}
the property in the second statement extends to that space as well. The corresponding isometry (bijection) is implemented by the adjoint operator $\Op L^\ast: L_2 \to \Spc H_\Op A\embedC \Spc S'$, which then also yields $L_2=\Op L^{\ast-1}(\Spc H_\Op A)$ where $\Op L^{\ast-1}$ is the stable inverse of $\Op L^\ast$ over its closed range $\Spc H_\Op A=\Op L^\ast(L_2)$.

When it exists, the {\em Green's operator} $\Op G: \Spc S \to L_2$ is the extension to
$\Spc S$ of the stable inverse ${\Op L^\ast}^{-1}\big|_{\Spc S_{\V p}}$, with the algebraic
property that $\Op L^\ast\Op G\{\varphi\}=\varphi$ for all $\varphi \in \Spc S$; on
$\Spc S_{\V p}$, this property is automatic. Its adjoint $\Op G^\ast=\Lop^{-1}: L_2 \to \Spc S'$
is then a right inverse of $\Lop$, which will resurface in the conditioned inverse
\eqref{Eq:gInverse} of Section \ref{Sec:NativeSpace}. Let us stress that no conditional
inverse of the normal operator is given by the naive composition $\Op G^\ast\Op G$, which
diverges on $\Spc S$ whenever $N_0>0$; the correct counterpart of the factorization
$\Op B=\Lop^\ast\Lop$ is the conditioned identity
$\Op A_{\V \phi}=\Op L^{-1}_{\V \phi}\Op L^{-1\ast}_{\V \phi}$ of \eqref{Eq:Factorization} in
Section \ref{Sec:NativeSpace}.

As for the primary isometry $\dot{\Spc H}_\Op L \to L_2$, it is implemented by the continuous extension of the operator $\Op L: \overline{(\Spc S,\|\Op L\cdot\|_{L_2})}=\dot{\Spc H}_\Op L \to L_2$. Since the elements of $\dot{\Spc H}_\Op L$ are equivalence classes, this requires that
$\Op L$ annihilates the null class $\Spc N_\V p=\Spc N_\Op B$. 
The complete characterization of the corresponding class of operators is provided by the following result of Part I (Theorem~\ref{I-Thm:PracticalChecklist}).

\begin{theorem}[Characterization of spline admissibility]
\label{Thm:PracticalChecklist}
Let $\Op L:\Spc S\to L_2$ be a continuous linear operator, and let
\begin{align*}
Q(\varphi)
\eqdef
\sup_{\psi\in\Spc S\backslash\{0\}}
\frac{|\langle\varphi,\psi\rangle|^2}{\|\Op L\psi\|_{L_2}^2}
\in[0,+\infty],
\qquad
\varphi\in\Spc S,
\end{align*}
denote the dual energy. The following four conditions are each necessary and, together,
sufficient for $\Op L$ to be $\V p$-admissible.
\begin{enumerate}
\item \emph{(Injectivity)} $\Op L$ is injective on $\Spc S$.
\item \emph{(Dense range)} $\Op L(\Spc S)$ is dense in $L_2$; equivalently,
$\ker(\Op L^\ast)=\{0\}$.
\item \emph{(Strict $\V p$-conditional positivity)}
\begin{align*}
Q(\varphi)<\infty
\quad\Longleftrightarrow\quad
\varphi\in\Spc S_{\V p}.
\end{align*}
\item \emph{(Sharpness)}
\begin{align*}
\|\Op L\psi\|_{L_2}
=\sup_{\varphi\in\Spc S_{\V p}\backslash\{0\}}
\frac{|\langle\psi,\varphi\rangle|}{\sqrt{Q(\varphi)}},
\qquad
\psi\in\Spc S.
\end{align*}
\end{enumerate}
When Conditions 1--4 hold, we have that $\Spc N_{\Op B}=\Spc N_{\V p}$ and that
$Q(\varphi)=|\varphi|^2_\Op A$ on $\Spc S_{\V p}$ for every admissible $\Op A$ in the norm
defined by \eqref{Eq:NormA}. 
\end{theorem}

Ultimately, Conditions 3 and 4 say that $\|\Op L\psi\|_{L_2}$ and
$\|{\Op L^\ast}^{-1}\varphi\|_{L_2}$ are dual norms for
$(\psi,\varphi) \in \Spc S \times \Spc S_\V p$.

\section{Native Space for $\Op L$}
\label{Sec:NativeSpace}

Let $\Lop$ be $\V p$-admissible in the sense of Definition \ref{Def:SplineAdmissible} 
and $(\Op A,\Op B)$, with $\Op B=\Lop^\ast\Lop$, the underlying $\V p$-conditional Riesz pair. Given some
biorthogonal system $(\V p, \V \phi)$ with $\V \phi \subset \Spc S$ and the projectors specified in Section \ref{Sec:VecFunctional}, we then construct the corresponding {\em
Gram operator}
\begin{align}
\label{Eq:Aphi}
\Op A_{\V \phi}\eqdef\Op P^\ast_{\V \phi}\Op A\Op P_{\V \phi}: \Spc S \to \Spc S'
\end{align}
which is symmetric and $\V p$-conditionally positive.
Part I \cite{Unser2026a} provides two
complementary descriptions of the induced Hilbert structure. The first one is the
\emph{augmented pair}
\begin{align}
\label{Eq:Aplus}
\Op A_+\eqdef \Op A_{\V \phi}+\Op R_{\V p},
\qquad
\Op B_+\eqdef \Op B+\Op R_{\V \phi},
\end{align}
which is a compatible Riesz pair for $(\Spc S,\Spc S')$ (Part I, Theorem~\ref{I-Thm:Augmentation}).
It generates the dual pair of Hilbert--Gelfand spaces
\begin{align}
\label{Eq:AugmentedPair}
\Spc H=\overline{(\Spc S,|\cdot|_{\Op A_+})}=\Spc H_{\Op A}\oplus \Spc N_{\V \phi},
\qquad
\Spc H'=\overline{(\Spc S,|\cdot|_{\Op B_+})}=\Spc H'_{\Op A}\oplus \Spc N_{\V p}.
\end{align}
The native space is the second one:
$\Spc H_\Lop=\Spc H'$. In particular, the dictionary between the notation of Part I and the one
used here is $\Spc H_{\Lop,\V \phi}=\Spc H'_{\Op A}$ and
$\Spc H'_{\Lop,\V \phi}=\Spc H_{\Op A}\eqdef\Lop^\ast(L_2)$, with the prime being always attached to
the dual of the space that carries the same subscript. 

The second description is the $L_2$-isometry induced by the factorization of the Gram operator as
\begin{align}
\label{Eq:Factorization}
\Op A_{\V \phi}=\Op L^{-1}_{\V \phi}\circ \Op L^{-1\ast}_{\V \phi}: \Spc S \to L_2 \to \Spc S',
\end{align}
which is the exact dual of the factorization $\Op B=\Lop^\ast\Lop$ of the normal operator.
The underlying operator there is the \emph{conditioned inverse}
\begin{align}
\label{Eq:gInverse}
\Op L^{-1}_{\V \phi}\eqdef \Op A_{\V \phi}\Lop^\ast=\Op P^\ast_{\V \phi}\Op G^\ast=(\Identity-\Op R_{\V p}\Op R_{\V \phi})\Op G^\ast: L_2 \to \Spc H'_{\Op A}
\end{align}
of Part I (Proposition~\ref{I-Prop:Factorization} and Corollary~\ref{I-Cor:GreenFactor}), where
$\Op G^\ast=\Op L^{-1}$ is the adjoint of the Green's operator $\Op G$ of
Section \ref{Sec:Preliminaries}; the last two expressions require the existence of that
operator, while the first one does not. The operator $\Op L^{-1}_{\V \phi}$ is the unique right-inverse of $\Lop$ that imposes
the boundary condition $\V \phi(f)=\V 0$.

Theorem \ref{Theo:NativeSpace} below combines these two descriptions into a single comprehensive characterization of $\Spc H_\Op L$.
The gain is twofold: the resulting description of $\Spc H_\Lop$ is directly usable in
applications, as it involves $\Lop$, $\Lop^\ast$ and $\Op L^{-1}_{\V \phi}$ only; and it holds for
a much broader range of biorthogonal systems, the requirement $\V \phi \subset \Spc S$ of
Part I being relaxed into $\V \phi \subset \Spc H'_\Lop$. 

\begin{theorem}[Native space associated with an admissible operator $\Lop$]
\label{Theo:NativeSpace}
A $\V p$-admissible operator $\Lop: \Spc S \to L_2$ admits a continuous extension
$\Lop: \Spc H_\Lop \to L_2$ that specifies its native space $\Spc H_\Lop$.
This
space has the following properties.
\begin{enumerate}
\item $\Spc H_\Lop$  is the Hilbert space induced by the inner product
\begin{align}
\label{Eq:InnerProd}\langle f, g\rangle_{\Spc H_\Lop}=\langle \Lop f, \Lop g \rangle_{L_2}+\sum_{n=1}^{N_0} \langle \phi_n,f\rangle\langle \phi_n,g\rangle
\end{align}
where $(\V \phi,\V p)$ is a biorthogonal system for $\Spc N_\V p$ with $\V \phi \subset \Spc S$ (universality); Item \ref{Item:NormEquiv} relaxes that condition.
\item \label{Item:DirectSum}Direct-sum decomposition: $\Spc H_\Lop=\Spc H_{\Lop, \V \phi} \oplus \Spc N_\V p$ with
\begin{align}
\Spc H_{\Lop, \V \phi}&=\{f \in \Spc H_\Lop: \V \phi(f)=\V 0\}=\Lop_{\V \phi}^{-1}(L_2)\\
\Spc N_\V p&=\{f \in \Spc H_\Lop: \Lop f=0\}={\rm span}\{p_n\}_{n=1}^{N_0}.
\end{align}
The primary Hilbert space $\Spc H_{\Lop, \V \phi}$ with
$\|f\|_{\Spc H_{\Lop, \V \phi}}=\|\Lop f\|_{L_2}$ is isomorphic to $L_2=\Lop(\Spc H_{\Lop, \V \phi})$, while $\Spc N_\V p$ is the effective null space of $\Lop$.
\item Continuous and dense embeddings (Hilbert--Gelfand property): $\Spc S \embedD \Spc H_\Lop \embedD  \Spc S'$. In particular, this enables us to specify $\Spc H_\Lop$ as $\Spc H_{\Lop}=\overline{(\Spc S,\|\cdot\|_{\Spc H_\Lop} )}$.
\item \label{Item:NormEquiv}Equivalent norms and topologies: As a set, $\Spc H_\Lop$ is independent of the choice of biorthogonal system $(\V \phi, \V p)$ subject to the constraint $\phi_1, \cdots,\phi_{N_0} \in \Spc H'_\Lop$. The norms induced by such systems are 
all topologically equivalent.

\item \label{Item:ExplicitRep} Explicit representation: For any $f \in \Spc H_\Lop$, there exist some unique $w \in L_2$, $f^\ast \in \Spc H'_{\Lop, \V \phi}=\Lop^\ast(L_2)$ and $p \in \Spc N_{\V p}$ such that
\begin{align}
\label{Eq:Sumdecomp1}
f&=\Op A_{\V \phi} \{f^\ast \}+ p=\Op L_{\V \phi}^{-1}\{ w\} + p
\end{align}
with $w=\Lop f$, $f^\ast=(\Lop^\ast\Lop)\{ f\}$ and $p=\Op R_{\V p}\Op R_{\V \phi}\{f\}$ where the linear operators 
$\Op L_{\V \phi}^{-1}: L_2 \to \Spc H_\Lop$ and $\Op A_{\V \phi}
: \Spc H'_{\Lop}\to \Spc H_\Lop$ are defined by \eqref{Eq:gInverse} and \eqref{Eq:Aphi}, respectively. 
\item \label{Item:RieszMaps}Explicit Riesz maps:
\begin{align}
\Op J_{\Spc H_\Lop}&=\Lop^\ast \Lop + \Op R_{\V \phi}: \Spc H_\Lop \to \Spc H_\Lop'\\
\Op J_{\Spc H'_\Lop}&=\Op J_{\Spc H_\Lop}^{-1}=\Op A_{\V \phi} + \Op R_{\V p}: \Spc H'_\Lop \to \Spc H_\Lop\\
\Op J_{\Spc H_\Op A}&=\Op A_{\V \phi}: \Spc H_\Op A \to \Spc H'_\Op A=\Spc H_{\Lop,\V \phi}.
\end{align}
\item Homogeneous part of the space:
\begin{align}
\dot{\Spc H}_{\Op B}=\overline{(\Spc S,\|\Op L\cdot\|_{L_2}) }=\Spc H_\Lop/\Spc N_\V p\cong \Spc H_{\Lop, \V \phi}.\end{align}
This quotient space is isometric to $\Spc H'_{\Lop,\V \phi}=\Spc H_\Op A=\overline{(\Spc S_\V p,|\cdot|_{\Op A}) }$ with the corresponding Riesz map being $\pi \Op A=\pi \Op A_{\V \phi}: \Spc H_\Op A \to \dot{\Spc H}_{\Op B}$. Unlike $\Spc H_{\Lop, \V \phi}$, the structure of $\dot{\Spc H}_{\Op B}=\Spc H_\Lop/\Spc N_\V p$ and $\Spc H_\Op A=(\Spc H_\Lop/\Spc N_\V p)'$ (core) is independent of $\V \phi$.
\end{enumerate}
\end{theorem}

\begin{proof}[Proof of Theorem \ref{Theo:NativeSpace}]
Let us assume first that $\V \phi \subset \Spc S$, so that the augmented pair \eqref{Eq:Aplus} of
Part I is available. Items 1--3, 6, and 7 are then transcriptions of the results of Part I,
while Items 4 and 5 are new.

\medskip
\emph{Inner product and embeddings (Items 1 and 3)}:
The augmented operator $\Op B_+=\Op B+\Op R_{\V \phi}$ is positive-definite on $\Spc S$ with
$|f|^2_{\Op B_+}=\|\Lop f\|^2_{L_2}+|\V \phi(f)|_2^2$, and $(\Op A_+,\Op B_+)$ is a compatible
Riesz pair (Part I, Theorem~\ref{I-Thm:Augmentation}). Theorem~\ref{I-Theo:HilbertPositiveDef} of
Part I then identifies $\Spc H_\Lop=\overline{(\Spc S,|\cdot|_{\Op B_+})}$ as a
Hilbert--Gelfand space, which yields the embeddings $\Spc S \embedD \Spc H_\Lop \embedD \Spc S'$
and the completion formula of Item 3. The polarization of $|\cdot|_{\Op B_+}$ gives the inner
product, where the factorization $\Op B=\Lop^\ast\Lop$ turns $\langle \Op B f, g\rangle$ into
$\langle \Lop f, \Lop g\rangle_{L_2}$; this is \eqref{Eq:InnerProd}.

\medskip
\emph{Direct-sum decomposition (Item 2)}:
The orthogonal decomposition $\Spc H_\Lop=\Spc H'=\Spc H'_{\Op A}\oplus \Spc N_{\V p}$ is Item 3
of Theorem~\ref{I-Thm:Augmentation} of Part I. Its infinite-dimensional component is
characterized by the boundary condition $\V \phi(f)=\V 0$, because
$\Spc H'_{\Op A}\subset \Spc S'_{\V \phi}$, and it coincides with $\Op L^{-1}_{\V \phi}(L_2)$
because $\Op L^{-1}_{\V \phi}$ maps $L_2$ isometrically onto $\Spc H'_{\Op A}$ (Part I,
Proposition~\ref{I-Prop:Factorization}, Item 1). This gives
$\Spc H_{\Lop, \V \phi}=\Spc H'_{\Op A}$ with $\|f\|_{\Spc H_{\Lop, \V \phi}}=\|\Lop f\|_{L_2}$.
The finite-dimensional component is the effective null space of $\Lop$, since the continuous
extension of $\Lop$ annihilates $\Spc N_{\Op B}=\Spc N_{\V p}$, and nothing else in
$\Spc H_\Lop$.

\medskip
\emph{Riesz maps (Item 6)}:
Item 4 of Theorem~\ref{I-Thm:Augmentation} of Part I states that
$\Op A_+=\Op A_{\V \phi}+\Op R_{\V p}$ and $\Op B_+=\Lop^\ast\Lop+\Op R_{\V \phi}$ are isometric
and mutually inverse, which identifies them as the Riesz maps $\Op J_{\Spc H'_\Lop}$ and
$\Op J_{\Spc H_\Lop}$. Since the duality pairing is block-diagonal with respect to the two direct
sums of \eqref{Eq:AugmentedPair}, the restriction of $\Op A_+$ to $\Spc H_{\Op A}$ is
$\Op A_{\V \phi}: \Spc H_{\Op A}\to \Spc H'_{\Op A}=\Spc H_{\Lop,\V \phi}$, which is the Riesz map
of $\Spc H_{\Op A}$. Its factorization $\Op A_{\V \phi}=\Op L^{-1}_{\V \phi}\Op L^{-1\ast}_{\V \phi}$
is \eqref{Eq:Factorization}.

\medskip
\emph{Homogeneous part (Item 7)}:
The first equality is the definition \eqref{Eq:HB} of the homogeneous space, since
$|\varphi|_{\Op B}=\|\Lop \varphi\|_{L_2}$ on $\Spc S$. Item 1 of
Theorem~\ref{I-Thm:Concretization} of Part I identifies $\Op P^\ast_{\V \phi}$ as an isometric
isomorphism $\dot{\Spc H}_{\Op B}\to\Spc H'_{\Op A}=\Spc H_{\Lop,\V \phi}$. Moreover, since the
direct sum of Item 2 is orthogonal, the quotient norm on $\Spc H_\Lop/\Spc N_{\V p}$ is attained
at $f_0=\Op P^\ast_{\V \phi}f$ with
$\min_{p\in\Spc N_{\V p}}\|f+p\|_{\Spc H_\Lop}=\|f_0\|_{\Spc H_\Lop}=\|\Lop f\|_{L_2}$, which
yields the isometry $\Spc H_\Lop/\Spc N_{\V p}\cong\Spc H_{\Lop,\V \phi}\cong\dot{\Spc H}_{\Op B}$.
The duality with $\Spc H_{\Op A}$ and the identification of the Riesz map $\pi\Op A$ constitute
Clause (ii) of Definition~\ref{Def:Intrinsic}, while the identity $\pi\Op A_{\V \phi}=\pi\Op A$
and the independence of $\dot{\Spc H}_{\Op B}$, $\Spc H_{\Op A}$, and $\pi\Op A$ from the choice
of $\V \phi$ are Items 2 and 3 of Theorem~\ref{I-Thm:Concretization}.

\medskip
\emph{Equivalence of norms (Item 4)}:
As an alternative to $(\V \phi,\V p)$, we consider another biorthogonal 
system $(\widetilde{\V \phi},\widetilde{\V p})$ with $\widetilde{\V p}=(\tilde p_n)$ some basis of $\Spc N_{\V p}$,
$\widetilde{\V \phi}=(\widetilde \phi_n)$ with $\widetilde \phi_1, \dots, \widetilde \phi_{N_0} \in \Spc H'_\Lop$ and $\langle \widetilde \phi_m, \widetilde p_n\rangle=\delta_{m,n}$ (biorthogonality). The corresponding inner product
induces the Hilbertian norm
\begin{align}
\|f\|_{\widetilde{\Spc H}_\Lop}=\left(\|\Lop f\|_{L_2}^2 + \sum_{n=1}^{N_0} |\langle \widetilde{\phi}_n, f \rangle|^2\right)^{\tfrac{1}{2}}.
\end{align}

Since two norms that are equivalent on $\Spc S$ yield the same completion, it suffices to
establish the two-sided bound \eqref{Eq:NormEquiv} below.

The upper bound follows from the duality bound \eqref{Eq:DualityBound}, which is
applicable because $\widetilde{\phi}_n \in \Spc H'_\Lop$. Together with
$\|\Lop f\|_{L_2}\le \|f\|_{\Spc H_\Lop}$, it gives
\begin{align*}
\|f\|^2_{\widetilde{\Spc H}_\Lop}\le \|\Lop f\|_{L_2}^2 + \sum_{n=1}^{N_0} \|\widetilde{\phi}_n\|^2_{\Spc H'_\Lop}\|f\|^2_{\Spc H_\Lop}\le \left(1 + \sum_{n=1}^{N_0} \|\widetilde{\phi}_n\|^2_{\Spc H'_\Lop}\right) \|f\|^2_{\Spc H_\Lop}.
\end{align*}
The lower bound cannot be obtained by exchanging the roles of the two systems, for the dual norm
$\|\cdot\|_{\widetilde{\Spc H}'_\Lop}$ is not available before the equivalence has been
established. We rely instead on the structure of $\Lop$. By Items 1 and 2, the operator
$\Lop: \Spc H_\Lop \to L_2$ is bounded and surjective with $\ker \Lop=\Spc N_{\V p}$, so that the
induced map $\Spc H_\Lop/\Spc N_{\V p}\to L_2$ is a continuous bijection between Banach spaces.
Its inverse is continuous as well, which provides a constant $c_\Lop$ such that
\begin{align}
\label{Eq:BoundedBelow}
\min_{p \in \Spc N_{\V p}}\|f+p\|_{\Spc H_\Lop}\le c_\Lop \|\Lop f\|_{L_2},\qquad f \in \Spc H_\Lop.
\end{align}
The subspace $\Spc H_{\Lop,\widetilde{\V \phi}}=\{f \in \Spc H_\Lop: \widetilde{\V \phi}(f)=\V 0\}$ is
closed, because the $\widetilde{\phi}_n$ are continuous on $\Spc H_\Lop$, and biorthogonality makes
it a topological complement of $\Spc N_{\V p}$. On such a complement, the minimum in
\eqref{Eq:BoundedBelow} is equivalent to the norm itself, which yields a constant $c'_\Lop$ with
$\|f_0\|_{\Spc H_\Lop}\le c'_\Lop\|\Lop f_0\|_{L_2}$ for every $f_0 \in \Spc H_{\Lop,\widetilde{\V \phi}}$.
We now split $f=f_0+\tilde p$ with
$\tilde p=\Op R_{\widetilde{\V p}}\Op R_{\widetilde{\V \phi}}f=\sum_{n=1}^{N_0} \tilde p_n \langle \widetilde{\phi}_n, f\rangle \in \Spc N_{\V p}$
and $f_0=f-\tilde p\in \Spc H_{\Lop,\widetilde{\V \phi}}$. Since $\Lop f_0=\Lop f$, the
triangle and Cauchy--Schwarz inequalities give
\begin{align*}
\|f\|_{\Spc H_\Lop}\le \|f_0\|_{\Spc H_\Lop} + \|\tilde p\|_{\Spc H_\Lop}
\le c'_\Lop \|\Lop f\|_{L_2} + C_{\widetilde{\V p}}\,|\widetilde{\V \phi}(f)|_2
\le \sqrt{{c'_\Lop}^2+C_{\widetilde{\V p}}^2}\;\|f\|_{\widetilde{\Spc H}_\Lop}
\end{align*}
with $C_{\widetilde{\V p}}=\big(\sum_{n=1}^{N_0}\|\tilde p_n\|^2_{\Spc H_\Lop}\big)^{\tfrac12}$. The
two bounds combine into
\begin{align}
\forall f \in \Spc H_\Lop:\quad A_1 \|f\|_{\Spc H_\Lop} \le \|f\|_{\widetilde{\Spc H}_\Lop} \le A_2 \|f\|_{\Spc H_\Lop} 
\label{Eq:NormEquiv}
\end{align}
with $A_1=\big({c'_\Lop}^2+C_{\widetilde{\V p}}^2\big)^{-\tfrac{1}{2}}$ and
$A_2=\left(1 + \sum_{n=1}^{N_0} \|\widetilde{\phi}_n\|^2_{\Spc H'_\Lop}\right)^{\tfrac{1}{2}}$.

\medskip
\emph{Explicit representation (Item 5)}:
Let $f \in \Spc H_\Lop$ and let us set $w=\Lop f \in L_2$,
$p=(\Identity-\Op P^\ast_{\V \phi})f=\Op R_{\V p}\Op R_{\V \phi}f \in \Spc N_{\V p}$ and
$f_0=\Op P^\ast_{\V \phi}f \in \Spc H_{\Lop,\V \phi}$, which is the decomposition of Item 2. The
inversion relation $\Op L^{-1}_{\V \phi}\Lop=\Op P^\ast_{\V \phi}$ of Part I
(Proposition~\ref{I-Prop:Factorization}, Item 1) then gives $f_0=\Op L^{-1}_{\V \phi}\{w\}$.
Likewise, $f^\ast=\Op B f=\Lop^\ast w \in \Lop^\ast(L_2)=\Spc H'_{\Lop,\V \phi}$, so that the
definition $\Op L^{-1}_{\V \phi}=\Op A_{\V \phi}\Lop^\ast$ in \eqref{Eq:gInverse} yields
$\Op A_{\V \phi}\{f^\ast\}=\Op A_{\V \phi}\Lop^\ast w=\Op L^{-1}_{\V \phi}\{w\}=f_0$. This
proves the two forms of \eqref{Eq:Sumdecomp1}. The triplet $(w,f^\ast,p)$ is unique
because the sum in Item 2 is direct, while $\Lop$ and $\Lop^\ast$ are injective on
$\Spc H_{\Lop,\V \phi}$ and $L_2$, respectively.

\medskip
\emph{Non-universal systems}:
Let us finally consider a system $(\widetilde{\V p},\widetilde{\V \phi})$ with
$\widetilde{\V \phi}\subset \Spc H'_\Lop$. The set $\Spc H_\Lop$ and its topology are the ones
that have just been constructed from a universal system, while Item 4 shows that the norm
induced by $\widetilde{\V \phi}$ is equivalent to $\|\cdot\|_{\Spc H_\Lop}$. All the statements
are therefore preserved, the splitting of Item 2 being the only ingredient that depends on the
choice of the system: the component
$\tilde p=\Op R_{\widetilde{\V p}}\Op R_{\widetilde{\V \phi}}\{f\}$ differs from
$p=\Op R_{\V p}\Op R_{\V \phi}\{f\}$, while $w=\Lop f$ is common to all splittings.
\end{proof}

For completeness, we record the property of the conditional inverse $\Op A$ that makes the
representation of the members of $\Spc H_\Lop$ independent of any biorthogonal system.

\begin{proposition}[Range of the conditional inverse]
\label{Prop:RangeA}
Let $\Lop$ be $\V p$-admissible with conditional inverse $\Op A$ of the normal operator
$\Op B=\Lop^\ast\Lop$, extended by continuity to $\Op A: \Spc H'_\Lop \to \Spc S'$. Let us
further assume that $\Op A$ is {\em consistent}, in the sense that
\begin{align}
\label{Eq:Consistency}
(\Lop^\ast\Lop)\{\Op A \varphi\}=\varphi \quad \mbox{for every } \varphi \in \Spc S \mbox{ such
that } \Op A\{\varphi\} \in \Spc H_\Lop,
\end{align}
where $(\Lop^\ast\Lop): \Spc H_\Lop \to \Spc H_{\Op A}$ is the extension of
Item \ref{Item:RieszMaps} of Theorem \ref{Theo:NativeSpace}. Then,
\begin{align}
\label{Eq:Afprop}
\Op A\{\nu\} \in \Spc H_\Lop
\quad \Leftrightarrow \quad
\nu \in \Spc H_{\Op A}=\Lop^\ast(L_2)
\quad \Leftrightarrow \quad
\V p(\nu)=\V 0 .
\end{align}
Consequently, every $f \in \Spc H_\Lop$ has a unique decomposition
\begin{align}
\label{Eq:DecomposeAgostic}
f=\Op A\{f^\ast\} + p, \qquad
f^\ast=(\Lop^\ast\Lop)\{f\}\in \Spc H_{\Op A}, \quad p \in \Spc N_{\V p},
\end{align}
which is the counterpart of Item \ref{Item:ExplicitRep} of Theorem \ref{Theo:NativeSpace} that
does not refer to any biorthogonal system.
\end{proposition}

\begin{proof}[Proof of Proposition \ref{Prop:RangeA}]
Let us fix a universal biorthogonal system $(\V \phi,\V p)$, so that
$\Spc H'_\Lop=\Spc H_{\Op A}\oplus \Spc N_{\V \phi}$ with $\Spc N_{\V \phi}\subset \Spc S$, and
let us recall that $\Spc H_{\Op A}=\{\nu \in \Spc H'_\Lop: \V p(\nu)=\V 0\}$, which is the second
equivalence in \eqref{Eq:Afprop}.

\medskip
\emph{Sufficiency}:
Let $\nu \in \Spc H_{\Op A}$. Since $\Op P_{\V \phi}$ acts as the identity on $\Spc H_{\Op A}$, we
have that $\Op A_{\V \phi}\{\nu\}=\Op A\{\nu\}-\Op R_{\V p}\Op R_{\V \phi}\Op A\{\nu\}$. The first
term of the right-hand side is in $\Spc H_{\Lop,\V \phi}$ by Item \ref{Item:ExplicitRep} of
Theorem \ref{Theo:NativeSpace}, while the second one is in $\Spc N_{\V p}$, which yields
$\Op A\{\nu\} \in \Spc H_{\Lop,\V \phi}\oplus \Spc N_{\V p}=\Spc H_\Lop$.

\medskip
\emph{Necessity}:
Let $\phi \in \Spc N_{\V \phi}\backslash\{0\}$ and let us assume that
$\Op A\{\phi\} \in \Spc H_\Lop$. The consistency requirement \eqref{Eq:Consistency} gives
$\phi=(\Lop^\ast\Lop)\{\Op A\phi\}$, while Item \ref{Item:ExplicitRep} of
Theorem \ref{Theo:NativeSpace} places the right-hand side in $\Spc H_{\Op A}$. Hence
$\phi \in \Spc H_{\Op A}\cap \Spc N_{\V \phi}=\{0\}$, which is a contradiction. A general
$\nu \in \Spc H'_\Lop$ splits as $\nu=\psi+\phi$ with $\psi \in \Spc H_{\Op A}$ and
$\phi \in \Spc N_{\V \phi}$; since $\Op A\{\psi\} \in \Spc H_\Lop$ by the previous step, the
hypothesis $\Op A\{\nu\} \in \Spc H_\Lop$ forces $\Op A\{\phi\} \in \Spc H_\Lop$, and hence
$\phi=0$.

\medskip
\emph{Decomposition}:
Item \ref{Item:ExplicitRep} of Theorem \ref{Theo:NativeSpace} gives
$f=\Op A_{\V \phi}\{f^\ast\}+q$ with $q \in \Spc N_{\V p}$ and $f^\ast=(\Lop^\ast\Lop)\{f\}$, while the first step
gives $\Op A_{\V \phi}\{f^\ast\}=\Op A\{f^\ast\}-p'$ with $p' \in \Spc N_{\V p}$. This is
\eqref{Eq:DecomposeAgostic} with $p=q-p'$. As for uniqueness, let us suppose that
$\Op A\{f^\ast_1\}+p_1=\Op A\{f^\ast_2\}+p_2$ with $f^\ast_1,f^\ast_2 \in \Spc H_{\Op A}$. Then,
$\Op A\{f^\ast_1-f^\ast_2\} \in \Spc N_{\V p}$, and the application of $(\Lop^\ast\Lop)$ gives
$f^\ast_1=f^\ast_2$.
\end{proof}

\begin{remark}
\label{Rem:Consistency}
The consistency requirement \eqref{Eq:Consistency} is automatically met for
$\varphi \in \Spc S_{\V p}$ (Part I, Theorem~\ref{I-Thm:Concretization}, Item 2), and it holds
for any conditional inverse of Green's type. In particular, when $\Lop$ is shift-invariant, the
canonical Fourier-domain inverse of Section \ref{Sec:RKHS} satisfies the stronger algebraic
identity $(\Lop^\ast\Lop)\Op A=\Identity$ on $\Spc S$, which implies \eqref{Eq:Consistency}.
\end{remark}

\begin{remark}[Two readings of the inversion identity]
\label{Rem:TwoReadings}
Let us contrast two ways of expressing that $\Op A$ inverts $\Op B=\Lop^\ast\Lop$. The first is
$\Op B\Op A=\Identity$ on $\Spc S_{\V p}$, understood through the canonical Hilbertian extension
$\Op B: \Spc H_\Lop \to \Spc H_{\Op A}$; this statement is well-formed for every conditional
inverse, and it holds for all of them. The second is the Green's-function version
$\Op B\Op A\{\varphi\}=\varphi$ for all $\varphi \in \Spc S$, which requires an interpretation
of $\Op B$ on the range $\Op A(\Spc S)\not\subset \Spc H_\Lop$. The canonical interpretation is
the extension by transposition: since $\Op B$ is symmetric, we set
$\langle \widetilde{\Op B} u, \psi\rangle \eqdef \langle u, \Op B\psi\rangle$ for every
$u \in \Spc S'$ such that the right-hand side is well-defined for all $\psi \in \Spc S$. This
extension involves no arbitrary choice and agrees with the Hilbertian one on $\Spc H_\Lop$,
since both satisfy $\langle \widetilde{\Op B} f,\psi\rangle=\langle \Lop f, \Lop \psi
\rangle_{L_2}$. The verification that $\Op A(\Spc S)$ lies in the domain of $\widetilde{\Op B}$
is immediate for differential operators and follows from degree counting for fractional ones.
Specifically, the conditional inverse for $\Lop=(-\Delta)^{\gamma/2}$ is $\Op A=(-\Delta)^{-\gamma}$, the fractional integrator or Riesz potential of order $2\gamma$.
For $\varphi \in \Spc S(\R^d)$, $(-\Delta)^{-\gamma}\{\varphi\}$ with $2\gamma \notin 2\N$ exhibits a growth
$|\V x|^{2\gamma-d}$---with a logarithmic correction when $2\gamma-d$ is
an even integer---that is offset by the algebraic decay $|\V x|^{-(d+2\gamma)}$ of $\Op B\{\psi\}$
induced by the singularity of the symbol at $\bw=\V 0$; the product decays like
$|\V x|^{-2d}$, which is integrable at infinity irrespective of $\gamma$. Under this reading,
the second version is a strictly stronger requirement than the first: it selects the
Green's-type conditional inverses, and it pins down $\Op A$ on $\Spc N_{\V \phi}$ up to
$\Spc N_{\Lop^\ast\Lop}$-components.
\end{remark}

\section{Linear Inverse Problems over Native Spaces}
\label{Sec:InverseProblems}

\subsection{Generalized Variational Splines}
\label{Sec:VarSpline}
Our variational splines are specified as solutions of inverse problems where a function (or a distribution) $f \in \Spc H_\Lop$ 
is probed through a set of linear measurement functionals $\V \nu=(\nu_1, \dots,\nu_M)$.
Since the recovery of $f$ from such measurements is generally ill-posed, one seeks the ``most regular'' solution $f_{\rm spline}\in \Spc H_\Lop$ that fits the data, with our notion of
regularity being tied to an admissible operator $\Lop$.
In the noise-free scenario, this is formulated as $f_{0}=\arg \min_{f \in \Spc H_\Lop} \|\Lop f\|^2_{L_2}$ s.t. $\M y=\V \nu(f)$.
By using a $\V p$-admissible operator, one privileges solutions that are in the null space $\Spc N_{\V p}$ of $\Lop$, which is one of the important specificities of splines.
When the measurements are corrupted by noise, one adapts the methodology by approximating the data within the tolerance of the noise, rather than fitting them exactly. In practice, this is done by reformulating the task as the minimization of a combined loss that is the sum of a data-fidelity term
and a regularizing ``spline energy.''

\subsubsection{Abstract Representer Theorem}

As a prerequisite for this type of optimization, it is helpful to recall
a classic representer theorem that applies to the simpler case where the regularization involves a Hilbertian norm
rather than the semi-norm associated to some $\V p$-admissible operator $\Lop$.
\begin{theorem}[Minimum-norm fitting in Hilbert space and Lagrange duality]
\label{Theo:MinNorm}
Given a vector $\M z=(z_m)\in \R^M$, a Hilbert space $\Spc H$ and a set of continuous linear functionals $\nu_1,\cdots,\nu_M\in \Spc H'$, we consider the dual pair of optimization problems
\begin{align}
\min_{f \in \Spc H} &\tfrac{1}{2}\|f\|^2_{\Spc H} \quad\mbox{ subject to } \quad \langle\nu_m,f\rangle=z_m,\  m=1,\dots,M\label{Eq:Ginterpol}\\
\max_{\V \alpha \in \R^M} &\V \alpha^\Tr \M z - \frac{1}{2}\big\|\sum_{m=1}^M \alpha_m \nu_m\big\|_{\Spc H'}^2. \label{Eq:DualProb}
\end{align}
If the $\nu_1,\cdots,\nu_M\in \Spc H'$ are linearly independent, then both problems have unique solutions, denoted by
$f_0\in \Spc H$ and $\M a=(a_m) \in\R^M$, which are isometrically linked by the Riesz map with $\nu^\ast_m=\Op J_{\Spc H'}\{\nu_m\}$. Specifically,
\begin{align}
\sum_{m=1}^M a_m \nu_m&=\Op J_{\Spc H}\{f_0\} \in \Spc H'%
\nonumber \\ \hspace*{1cm} \quad \Leftrightarrow &\quad f_0=\Op J_{\Spc H'}\left\{\sum_{m=1}^M a_m \nu_m\right\}=\sum_{m=1}^M a_m \nu^\ast_m \in \Spc H.
\end{align}
\end{theorem}
The proof of Theorem \ref{Theo:MinNorm}, which solely rests on the existence of the Riesz map, can be found in Appendix \ref{App:RepTheorem}.
The power of this result
is that it enables us to reformulate a constrained infinite-dimensional minimization problem as an unconstrained finite-dimensional optimization \big(Eq.~\eqref{Eq:DualProb}\big) that can then be solved using standard vector calculus, as shown in the first part of the proof.

\subsubsection{Characterization of Variational Splines}
For the cases of our interest where the regularization involves a semi-norm, we need to be a bit more careful
to ensure that the minimization problem is well-posed.

\begin{definition}[$\Lop$-compatible measurements]
\label{Def:AdmisPoints} Let $\Spc H_\Lop$ be the native Hilbert space of a $\V p$-admissible operator $\Lop$ with
$\V p=(p_1,\dots,p_{N_0})$.
We say that the linear measurement operator $\V \nu=(\nu_m): \Spc H_\Lop \to \R^M$ is \emph{$\Lop$-compatible} if
\begin{enumerate}
\item  the $\nu_m$ are linearly independent with $\nu_m \in \Spc H_\Lop'$ (continuity) for all $m$;
\item the matrix $\M P=\V \nu(\V p) \in \R^{M \times N_0}$ with $[\M P]_{m,n}=\langle \nu_m, p_n\rangle$ is of rank $N_0$. 
\end{enumerate}
\end{definition}

The rank condition is precisely the hypothesis under which the least-squares fitting problem is well posed over the null space $\Spc N_{\V p}$ of $\Lop$. Indeed, by writing $p=\sum_{n=1}^{N_0} b_n p_{n}$, we have that
\begin{align}
\label{Eq:LSpol}
\min_{p\in \Spc N_{\V p}} \sum_{m=1}^M |y_m-\langle \nu_m, p\rangle|^2=
\min_{\M b \in \R^{N_0}} \|\M y- \M P \M b\|^2.
\end{align}
The minimum is achieved for $\M b=(b_n)=(\M P^\Tr\M P)^{-1}\M P^\Tr \M y$ where the rank condition guarantees the invertibility of the normal matrix $(\M P^\Tr\M P) \in \R^{N_0 \times N_0}$. 

We now have all the elements needed to characterize the solutions of a broad class of variational data-fitting problems in which the regularization is given by a spline energy.  Although the optimization is posed over an infinite-dimensional native space, a remarkable and unifying result emerges: all solutions share the same parametric form 
\eqref{Eq:f0lin2}. Specifically, they are generalized splines that belong to a finite-dimensional reconstruction space whose basis functions depend exclusively on the pair 
$(\Lop, \V \nu)$.

\begin{theorem}[Representer theorem for generalized variational splines]
\label{Theo:GeneralRepSemiHilbert}
Let $\Lop: \Spc S \to L_2$ be a $\V p$-admissible operator with native space $\Spc H_\Lop$ and
conditional inverse $\Op A$, let $\V \nu=(\nu_m): \Spc H_\Lop \to \R^M$ with $M>N_0$ be
$\Lop$-compatible in the sense of Definition \ref{Def:AdmisPoints}, and let
$E: \R^M \times \R^M \to \R_{\ge0}\cup \{+\infty\}$ be lower-semicontinuous, coercive, and
convex in its second argument. Then, for any fixed $\M y\in \R^M$ and $\lambda \in \R_{\ge 0}$,
the solution of the optimization problem
\begin{align}
\label{Eq:GenericOptimizationProb2}
\arg \min_{f \in \Spc H_\Lop} J(f|\M y,\lambda) \quad\mbox{ with }\quad J(f|\M y,\lambda)=E\big(\M y, \V \nu(f)\big)+ \lambda \|\Lop f\|^2_{L_2}
\end{align}
is unique and admits the
linear parametrization 
\begin{align}
\label{Eq:f0lin2}
f_\lambda=\sum_{m=1}^{M-N_0} \tilde a_m \widetilde{\psi}_m +\sum_{n=1}^{N_0} b_{n} p_n,
\end{align}
with expansion coefficients $(\widetilde{\V a},\V b)\in \R^{M}$ and 
fixed basis functions $p_n\in \Spc N_\V p$ and $\widetilde{\psi}_m=\Op A_{\widetilde{\V \phi}}\{\widetilde{\nu}_m\}\in \Spc H_{\Lop,\widetilde{\V \phi}}$, where
$\Op A_{\widetilde{\V \phi}}$ is the Riesz map $\Spc H_{\Op A} \to \Spc H'_{\Op A}=\Spc H_{\Lop,\widetilde{\V \phi}}$ 
and  $\widetilde{\V \nu}=(\widetilde{\nu}_m)=\widetilde{\M Q}^\Tr \V \nu$
is the complementary measurement basis specified in Lemma \ref{Prop:DirectSum}. 
Alternatively, the solution has an overparametrized representation as
\begin{align}
\label{Eq:f0lin2b}
f_\lambda=\sum_{m=1}^{M} a_m \psi_m +\sum_{n=1}^{N_0} b_{n} p_n\quad  \mbox{ s.t. }\quad
\M P^\Tr \M a=\M 0
\end{align}
with $\V a\in \R^{M}$ and $\psi_m=\Op A\{\nu_m\}$. The connection between the two parametrizations is
 $\M a=\widetilde{\M Q}\widetilde{\V a}$ and $\widetilde{\V a}=\M Q^\Tr \M a$ where the matrices $\widetilde{\M Q}, \M Q \in \R^{M \times (M-N_0)}$ are specified in Lemma \ref{Prop:DirectSum}.
\end{theorem}
Our approach for solving \eqref{Eq:GenericOptimizationProb2} is to select a 
biorthogonal system $(\widetilde{\V \phi},\V p)$ matched to the measurement functionals that enables us to decouple the two components of the solution. The construction of such a system is always possible as asserted by Lemma \ref{Prop:DirectSum} whose proof is given in Appendix \ref{App:MeasurementSpace}.

\begin{lemma}[Direct-sum decomposition of the measurement space]
\label{Prop:DirectSum}
Let $\V \nu=(\nu_1,\dots,\nu_M)$ be a vector of linear functionals with $\nu_m \in \Spc H'_\Lop$
and let $\V p=(p_1,\dots,p_{N_0})$ be a basis of $\Spc N_{\V p}$ such that the matrix
$\M P=\V \nu(\V p) \in \R^{M \times N_0}$ is of rank $N_0$.
Then, one can always find three matrices $\widetilde{\M P}\in \R^{M \times N_0}$, $\M Q\in \R^{M \times (M-N_0)}$, and $\widetilde{ \M Q} \in \R^{M \times (M-N_0)}$ 
such that 
\begin{align}
\label{Eq:Biortho2}
\left[
\begin{array}{c}
 \widetilde{\M Q}^\Tr\\
    \widetilde{\M P}^\Tr
 \end{array}
\right] \left[
\begin{array}{cc}
 {\M Q} &     {\M P} 
 \end{array}
\right]
=\M I_M.
\end{align}
Based on these matrices, $\V \nu$ has the unique, reversible decomposition
\begin{align}
\label{Eq:Decompnu}
\V \nu
= {\M Q} \widetilde{\V \nu} + {\M P} \widetilde{\V \phi},
\end{align}
where
\begin{align}
\label{Eq:ReducedBasis}
\widetilde{\V \nu}&=(\widetilde{\nu}_1,\dots, \widetilde{\nu}_{M-N_0})=\widetilde{\M Q}^\Tr\V \nu  \in (\Spc N_{\V \nu} \cap \Spc H_\Op A)^{M-N_0}\\
\label{Eq:ReducedBasis2}\widetilde{\V \phi}&=(\widetilde{\phi}_1,\dots, \widetilde{\phi}_{N_0})=\widetilde{\M P}^\Tr\V \nu
\in (\Spc N_{\V \nu})^{N_0},
\end{align}
with $\widetilde{\V \nu}(\V p)=\M 0$ (orthogonality) and $\widetilde{\V \phi}(\V p)=\widetilde{\M P}^\Tr\M P=\M I_{N_0}$ (biorthonormality).
\end{lemma}

Formally, the result in Lemma \ref{Prop:DirectSum} is equivalent to the direct-sum decomposition of the measurement space
as $\Spc N_{\V \nu}=\Spc N_{\widetilde{\V \nu}} \oplus \Spc N_{\widetilde{\V \phi}}$, where $\Spc N_{\widetilde{\V \nu}}$ annihilates $\Spc N_\V p$ and where $\widetilde{\V \phi}(\V p)=\M I_{N_0}$ (biorthogonality).

\begin{proof}[Proof of Theorem \ref{Theo:GeneralRepSemiHilbert}]
Equipped with such a problem-specific biorthogonal system $(\widetilde{\V \phi},\V p)$, we invoke Theorem \ref{Theo:NativeSpace} to decompose $f \in \Spc H_\Lop$ as $f=\tilde f + \tilde p$ with $\tilde f \in \Spc H_{\Lop, \widetilde{\V \phi}}$ and $\tilde p=\sum_{n=1}^{N_0} b_n p_n \in \Spc N_\V p$. In this configuration, $\V \nu=\M Q \widetilde{\V \nu}+ \M P \widetilde{\V \phi}$ with $\widetilde{\V \nu}=\widetilde{\M Q}^\Tr \V \nu \in \Spc H_\Op A^{M-N_0}$ and $\widetilde{\V \phi}=\widetilde{\M P}^\Tr \V \nu$ where the matrices $\M P,\widetilde{\M P}\in \R^{M \times N_0}$ and $\widetilde{\M Q}, \M Q \in \R^{M \times (M-N_0)}$ satisfy the invertibility relation \eqref{Eq:Biortho2}. Consequently,
\begin{align*}
\V \nu(f)=\big(\M Q \widetilde{\V \nu}+ \M P \widetilde{\V \phi}\big)(\tilde f + \tilde p)=\M Q\widetilde{\V \nu}(\tilde f)+ \M P \M b
\end{align*}
where we have made use of the underlying biorthogonality properties; namely, (i) $\widetilde{\V \phi}(\tilde f)=\V 0$  for all $\tilde f \in \Spc H_{\Lop, \widetilde{\V \phi}}$, (ii) $\widetilde{\V \nu}(p)=\V 0$  for all $p \in \Spc N_\V p$ (because $\widetilde{\V \nu}(\V p)=\widetilde{\M Q}^\Tr\V \nu(\V p)=\widetilde{\M Q}^\Tr\M P=\M 0$), and (iii) $\langle \widetilde{\phi}_n, \sum_{m=1}^{N_0} b_m p_m\rangle= b_n$. Since $\|\Lop f\|^2_{L_2}=\|\Lop \tilde f\|^2_{L_2}$, this enables us to rewrite the optimization problem \eqref{Eq:GenericOptimizationProb2} as
\begin{align}
\min_{(\tilde f, \M b)\in \Spc H_{\Lop, \widetilde{\V \phi}}\times  \R^{N_0}}& J(\tilde f, \M b) \nonumber\\ \mbox{ where }&
J(\tilde f, \M b)=E\Big(\M y, \M Q\widetilde{\V \nu}(\tilde f)+ \M P \M b\Big) + \lambda \|\Lop \tilde f\|^2_{L_2}.\label{Eq:Jmin}
\end{align}
Next, we observe that the functional
$\M b \mapsto J(\tilde f, \M b)$ with $\tilde f$ fixed is lower-semi\-continuous (l.s.c.), coercive (because $\M P$ is of full column rank) and convex on $\R^{N_0}$.
Likewise, $\tilde f \mapsto J(\tilde f, \M b)$ with $\M b$ fixed is l.s.c. (because of the continuity of $\V \nu: \Spc H\to \R^M$), coercive and {\em strictly convex} on $\Spc H_{\Lop, \widetilde{\V \phi}}$ because of the presence of the second term ($L_2$-norm). This implies that the global functional $(\tilde f , \M b)\mapsto J(\tilde f, \M b)$ is l.s.c., coercive and strictly convex, which ensures that the minimizer of \eqref{Eq:Jmin} exists and is unique \cite{Ekeland1999}.

This minimizer is denoted by $(\tilde f_\lambda, \M b_\lambda)$ and is such that
\begin{align}
\M Q\widetilde{\V \nu}(\tilde f_\lambda)+ \M P \M b_\lambda=
 \left[
\begin{array}{cc}
 {\M Q}  &  {\M P} 
 \end{array}
\right] \left[
\begin{array}{c}
\widetilde{\V \nu}(\tilde f_\lambda)  \\
\M b_\lambda 
 \end{array}
\right]
=\M z_0 \in \R^M,
\end{align} which fixes the first term of the criterion to some constant $E(\M y,\M z_0)$. By exploiting the invertibility of the underlying matrix, we get that
\begin{align}
\left[
\begin{array}{c}
\widetilde{\V \nu}(\tilde f_\lambda)  \\
\M b_\lambda 
 \end{array}
\right]=
 \left[
\begin{array}{cc}
 {\M Q}  &  {\M P} 
 \end{array}
\right]^{-1} \M z_0 =\left[
\begin{array}{c}
 \widetilde{\M Q}^\Tr \M z_0\\
    \widetilde{\M P}^\Tr \M z_0
 \end{array}
\right],
\end{align}
which directly yields $\M b_\lambda=  \widetilde{\M P}^\Tr \M z_0 \in \R^{N_0}$. As for $\tilde f_\lambda$, we identify it as the solution of the simpler generalized interpolation problem
\begin{align}
\tilde f_\lambda=\arg \min_{\tilde f \in \Spc H_{\Lop,\widetilde{\V \phi}}} \|\Lop \tilde f\|_{L_2} \ \mbox{ s.t.} \ \  
\widetilde{\V \nu}(\tilde f)= \widetilde{\M Q}^\Tr \M z_0 \in \R^{M-N_0},
\label{Eq:geninterpol2}
\end{align}
where $\|\Lop \tilde f\|_{L_2}$ coincides with the norm of the Hilbert space $\Spc H_{\Lop,\widetilde{\V \phi}}$ (see Theorem \ref{Theo:NativeSpace}, Item \ref{Item:DirectSum}). We can now invoke the Riesz representer Theorem \ref{Theo:MinNorm} to obtain the parametric form of the solution as
\begin{align}
\label{Eq:SolutionParam2}
\tilde f_{\lambda}=\Op A_{\widetilde{\V \phi}}\left\{\sum_{m=1}^{M-N_0} \tilde a_m \widetilde{\nu}_m\right\} =\sum_{m=1}^{M-N_0} \tilde a_m \Op A_{\widetilde{\V \phi}}\{\widetilde{\nu}_m\}
\end{align}
where $\Op A_{\widetilde{\V \phi}}$ is the Riesz map $\Spc H_\Op A \to \Spc H_{\Lop,\widetilde{\V \phi}}$, and where
$\tilde f^\ast_{\lambda}=\sum_{m=1}^{M-N_0} \tilde a_m \widetilde{\nu}_m$ is the Hilbert conjugate of
$\tilde f_{\lambda}$. Since $\widetilde{\psi}_m=\Op A_{\widetilde{\V \phi}}\{\widetilde{\nu}_m\}$, the
right-hand side of \eqref{Eq:SolutionParam2} is precisely the first sum of \eqref{Eq:f0lin2}. The
desired result is obtained by adding back the optimal null-space component
$\tilde p=\sum_{n=1}^{N_0} b_n p_n \in \Spc N_\V p$ with coefficient $\M b_\lambda$.

To get the more standard overparametrized form \eqref{Eq:f0lin2b}, we make use of
Proposition \ref{Prop:RangeA}, which gives $\Op A\{\tilde f^\ast_\lambda\} \in \Spc H_\Lop$ with
$\Op A\{\tilde f^\ast_\lambda\}-\Op A_{\widetilde{\V \phi}}\{\tilde f^\ast_\lambda\} \in \Spc N_{\V p}$;
the null-space coefficients $\M b$ of \eqref{Eq:f0lin2b} therefore differ from those of
\eqref{Eq:f0lin2} by that component. We then note that 
$\tilde f^\ast_{\lambda}=\widetilde{\M a}^\Tr \widetilde{\V \nu}=\widetilde{\M a}^\Tr \widetilde{\M Q}^\Tr\V \nu=\M a^\Tr \V \nu$ where $\M a=\widetilde{\M Q} \widetilde{\M a} \in \R^M$. The required orthogonality condition $\V p(\tilde f^\ast_{\lambda})=0$ is met implicitly because
$\M P^\Tr\M a=\M P^\Tr\widetilde{\M Q}\widetilde{\M a}=\V 0$. Conversely, we have that
$\widetilde{\M a}=\M Q^\Tr \M a$ if and only if $\M P^\Tr \M a=\M 0$. The matrix form of these equivalences is
\begin{align}
\M a = \left[
\begin{array}{c c}
 \widetilde{\M Q}   &  \widetilde{\M P}
 \end{array}
\right] \left[
\begin{array}{c}
\widetilde{\M a}\\
 \M 0
 \end{array}
\right] \quad \Leftrightarrow \quad \left[
\begin{array}{c}
\widetilde{\M a}\\
 \M 0
 \end{array}
\right] = \left[
\begin{array}{c}
\M Q^\Tr\\
   \M P^\Tr
 \end{array}
\right] \M a,
\end{align}
which comes as a direct consequence of \eqref{Eq:Biortho2}.
\end{proof}

One obtains an exact fit, as in \eqref{Eq:GsplineInt}, either by letting $\lambda\to 0$ or by selecting $E$ to be a barrier function that returns
$0$ if $\V \nu(f)=\M y$ and $+\infty$ otherwise. At the other end of the scale, for $\lambda\to \infty$, one gets a linear regression with $f_\infty \in \Spc N_{\V p}$.

Theorem \ref{Theo:GeneralRepSemiHilbert} is a variation around a classic theme in spline theory
that goes back to Schoenberg and de Boor \cite{Schoenberg1964,deBoor1966}. The present form is
probably the most general one to date, as it applies to a remarkably broad family of native
spaces $\Spc H_\Lop$ and to arbitrary continuous linear measurements
$\V \nu: \Spc H_\Lop \to \R^M$. While the two parametrizations are equivalent, they each have
their specificities. The conventional form \eqref{Eq:f0lin2b} does not require the choice of any
boundary functionals $\widetilde{\V \phi}$; the price is the orthogonality constraint
$\M P^\Tr\M a=\M 0$ ($N_0$ additional linear equations), which tames the basis functions
$\psi_m=\Op A\{\nu_m\}$, which, in general\footnote{The condition $\V p(\nu_m)=\V 0$ cannot be
met for all $m$ because it contradicts the rank requirement in
Definition \ref{Def:AdmisPoints} that makes the problem well-posed.}, are not included in
$\Spc H_\Lop$ \big(see \eqref{Eq:Afprop} in Proposition \ref{Prop:RangeA}\big): the constraint
produces the cancellation of the improper trends at infinity that ultimately ensures that
$f_\lambda \in \Spc H_\Lop$. The adapted form \eqref{Eq:f0lin2}, by contrast, is explicit and
inherently stable, in the sense that
$\widetilde \psi_m=\Op A_{\widetilde{\V \phi}}\{\widetilde{\nu}_m\} \in \Spc H_\Lop$, at the
cost of the construction of the basis in Lemma \ref{Prop:DirectSum}. This non-standard
parametrization is also the key for making the connection with a corresponding class of
generalized Gaussian processes, which would otherwise be ill-defined.

\subsection{Link with RKHS and Conventional Splines}
\label{Sec:RKHS}
We can make the link with reproducing-kernel Hilbert spaces when $\Spc H$ is a Hilbert space of continuous functions on $\R^d$
on which the sampling functionals are well-defined, with $\nu_m=\delta_{\V x_m} \in \Spc H'$ for any $\V x_m \in \R^d$.
In that setting with $\Spc N_{\V p}=\{0\}$, the {\em reproducing kernel} $h: \R^d \times \R^d \to \R$ is the Schwartz kernel (a.k.a.\ generalized impulse response) of the symmetric, positive-definite operator $\Op A$ of Theorem \ref{Theo:NativeSpace}, whose action is described explicitly by the integral formula
\begin{align}
\Op A\{\varphi\}(\V x)=\int_{\R^d} h(\V x,\V y) \varphi(\V y) \dint \V y.
\end{align}

If we now consider the corresponding interpolation/approximation problem \eqref{Eq:GenericOptimizationProb2} with $\langle\nu_m, f\rangle=\langle \delta_{\V x_m}, f\rangle=f(\V x_m)$
at distinct locations $\V x_1, \cdots,\V x_M \in \R^d$, we find that the basis functions in \eqref{Eq:f0lin2} are given by
$\psi_m=\Op A\{\delta_{\V x_m}\}=h(\cdot,\V x_m)$.
This yields $f_\lambda(\V x)=\sum_{m=1}^M a_m h(\V x,\V x_m)$, which has the form of the kernel estimators used in machine learning.

When $\Lop$ is linear shift-invariant (LSI), then its Green's operator and its conditional inverse
are shift-invariant as well, and both act by convolution with
$\Lop^{-1}: \varphi \mapsto \rho_\Lop \ast \varphi$ and
$\Op A=(\Lop^\ast\Lop)^{-1}: \varphi \mapsto \rho_{\Lop^\ast\Lop} \ast \varphi$, where
$\rho_\Lop=\Lop^{-1}\{\delta\}$ and $\rho_{\Lop^\ast\Lop}=(\Lop^\ast\Lop)^{-1}\{\delta\}$ are the
Green's functions of $\Lop$ and $\Lop^\ast\Lop$, respectively.
Since LSI operators are diagonalized by the Fourier transform, with
$\Op T_{\rm LSI}\{\ee^{\jj \langle \cdot, \bw\rangle}\}=\widehat {T}_{\rm LSI}(\bw)\,\ee^{\jj \langle \cdot, \bw\rangle}$
where $\widehat {T}_{\rm LSI}=\Fourier\big\{\Op T_{\rm LSI}\{\delta\}\big\}$ is the frequency
response, we readily deduce that the Schwartz kernel of
$\Op A=(\Lop^\ast\Lop)^{-1}=\Op A_{\rm LSI}$ has the convolutional structure
\begin{align}
\label{Eq:LSIKernel}
h_{\rm LSI}(\V x,\V y)=\rho(\V x- \V y), \quad \rho=\Fourier^{-1}\left\{\frac{1}{|\widehat L|^2}\right\}
\end{align}
where $\widehat L(\bw)$ (resp., $|\widehat L(\bw)|^2$) is the frequency response of $\Lop$ (resp., of $\Lop^\ast\Lop$).
Here, $\rho=\rho_{\Lop^\ast\Lop}$ is the symmetric Green's function of $(\Lop^\ast\Lop)$ whose existence (as well as the well-posedness of the underlying inverse Fourier transform) is guaranteed from the admissibility assumption. The spectral representation of $\Op A_{\rm LSI}$ with $\widehat{A}_{\rm LSI}(\bw)=\frac{1}{|\widehat L(\bw)|^2}\ge 0$ also suggests that $\Op A_{\rm LSI}$ is positive semi-definite by default (resp., positive definite if $\frac{1}{|\widehat L(\bw)|^2}>0$ for all $\bw \in \R^d$)
under some implicit assumption of integrability;
e.g., $1/\widehat L\in L_2(\R^d)$ which ensures that the inverse operator $\Lop^{-1}: L_2 \to \Spc H_\Lop$ is well-defined and that the Green's function $\rho_{\Lop^\ast\Lop}: \R^d \to \R$ is continuous.

It is instructive to detail these computations for the representative example of the derivative
$\Dop=\frac{\dint}{\dint x}$, whose symbol is
$\widehat D(\omega)=\jj \omega=\Fourier\{\delta'\}(\omega)$. This operator is associated with the
kernel $\rho_{\Dop^\ast\Dop}(x)=-\tfrac{1}{2}|x|=\Fourier^{-1}\{\frac{1}{|\omega|^2}\}(x)$, with
the corresponding conditional inverse being
$\Op A_\Op D=(\Op D^\ast\Op D)^{-1}: \varphi \mapsto \int_\R -\tfrac{1}{2}|\tau | \varphi(\cdot - \tau)\dint \tau$.
The latter is symmetric and $(1)$-conditionally positive because
$\langle \Op A_\Op D \varphi,\varphi\rangle=\int_\R \frac{|\hat \varphi(\omega)|^2}{\omega^2}\dint \omega>0$
for all $\varphi \in \Spc S(\R)\backslash \{0\}$ such that
$\langle 1, \varphi\rangle=\hat \varphi(0)=0$, where $\hat \varphi=\Fourier\{\varphi\}$ is the
Fourier transform of $\varphi$. Let us note that the latter orthogonality condition neutralizes the
singularity of the Fourier-domain integral at the origin. This also enables us to assert that the
map $\Dop^\ast|_\Spc S=-\Op D: \Spc S \to \Spc S_{(1)}$ is surjective with
$\Spc S_{(1)}=\Dop^\ast(\Spc S)\subset\Dop^\ast(L_2)$, which confirms the admissibility of $\Dop$.
The prototypical family of regular shift-invariant operators that are $\V p$-admissible with a null
space formed of polynomials is that of the $m$th-derivative operators
$\Dop^m: \Spc S(\R) \to \Spc S(\R)$, whose relevant components are listed in
Table \ref{Tab:Derivatives}.

\begin{table}
\setlength{\tabcolsep}{3pt}
\begin{tabular}{p{0.6cm}p{2.6cm}p{3.4cm}p{1.2cm}p{3.2cm}}
 \hline\\[-2ex]
    $\Lop$ &$\rho_{\Lop}$ & $\rho_{\Lop^\ast\Lop}$ &$N_0$ & $\{(p_n,\phi_n)\}_{n=1}^{N_0}$\\[1ex]
    \hline\\[1ex]
    $\Dop$ & $\Indic_{[0,\infty)}(x)=x_+^0$ & $-\tfrac{1}{2}|x|$ & $1$ & $\left\{\big(p_1(x)=1,\ \phi_1=\delta\big)\right\}$ \\[2ex]
     $\Dop^m$ & $\frac{x_+^{m-1}}{(m-1) !}$ & $\tfrac{(-1)^m}{2}\frac{|x|^{2m-1}}{(2m-1)!}$ & $m$ & $\left\{(\frac{x^{n-1}}{(n-1)!}\,,\delta^{(n-1)})\right\}$ \\   
\\
     $\Dop^\alpha$ & $\frac{x_+^{\alpha-1}}{\Gamma(\alpha)}$ & $ \tfrac{-1}{2 \sin\big(\frac{\pi }{2} (2\alpha-1)\big)}\frac{|x|^{2\alpha-1}}{\Gamma(2\alpha)}$ & $\lfloor \alpha+\tfrac{1}{2}\rfloor$ & $\left\{(\frac{x^{n-1}}{(n-1)!}\,,\delta^{(n-1)})\right\}$\\[2ex]        \hline
\end{tabular} \caption{Differential operators encountered in spline theory with associated Green's functions
and biorthogonal systems. The derivative operators here are either ordinary with $m \in \N$ or fractional with $\alpha \in \R^+$ and the Fourier symbol of $\Op D^\alpha$ being $(\jj \omega)^\alpha$ \cite{Unser2000a}. The size of the null space follows the general rule of Part I, which identifies the attainable null space of a symbol of order $\gamma$ on $\R^d$ as the space $\Pi_{k_\gamma}$ of polynomials of degree at most $k_\gamma=\lfloor \gamma-\tfrac{d}{2}\rfloor$ (Part I, Theorem \ref{I-Thm:PracticalChecklist} and Proposition \ref{I-Prop:MultiplierConstruction}); for $d=1$ and $\gamma=\alpha$, this gives $N_0=k_\alpha+1=\lfloor \alpha+\tfrac{1}{2}\rfloor$, which is the integer that is nearest to $\alpha$ and which reduces to $m$ for $\alpha=m \in \N$. The half-integer orders $\alpha \in \N+\tfrac{1}{2}$ are excluded, for the constant in $\rho_{\Lop^\ast\Lop}$ is then singular and the Green's function acquires a logarithmic factor.
\label{Tab:Derivatives} }
\end{table}

The biorthogonal systems of Table \ref{Tab:Derivatives} deserve a comment, for they are no longer
non-universal as in Part I: the functionals $\phi_n=\delta^{(n-1)}$ are in $\Spc H'_\Lop$, but not in $\Spc S$.
This is precisely the configuration that Item \ref{Item:NormEquiv} of
Theorem \ref{Theo:NativeSpace} covers, and it is what makes the leading example of the theory
compatible with our description of the native space.

The LSI structure of the kernel in \eqref{Eq:LSIKernel} with $\rho=\Fourier^{-1}\{\widehat A_{\rm LSI}\}$ is also achievable by less constrained configurations that do not explicitly involve a regularization operator $\Lop$. Such cases are covered by a standard result in interpolation theory: a shift-invariant kernel $h_{\rm LSI}(\V x,\V y)=\rho(\V x- \V y)$ is the reproducing kernel of a Hilbert space if and only if $\Fourier\{\rho\}=\widehat A_{\rm LSI}$ (the Fourier transform of $\rho$) is a non-negative finite Radon measure on $\R^d$. The latter statement is the RKHS variant of Bochner's theorem \cite{Berlinet2004}, which is itself a special case of Minlos' theorem, as discussed in Part III \cite{Unser2026c} with $\Spc X=\R^d$. In particular, if $\widehat A_{\rm LSI}$ can be identified as a non-negative measurable function on $\R^d$, then the condition reduces to $\widehat A_{\rm LSI} \in L_1(\R^d)$ (integrability). Bochner's theorem then implies that the resulting function $\rho: \R^d \to \R$ is continuous, and positive (semi-)definite in the sense of Part III \cite{Unser2026c}, with $\Spc X=\R^d$.

The spectral representation of $\Lop$ also gives insight into the more challenging configurations where the null space of $\Lop$ is non-trivial. Basically, the only way to achieve a finite-dimensional null space is to
annihilate a finite number of eigenvalues, which, for LSI operators, translates into $\widehat{L}(\bw)$ having a finite number of zeros.
These zeros can be concentrated at the same location, as in the case of
the $n$th derivative $\Op D^n$, whose frequency response $\widehat{D}^n(\omega)=(\jj \omega)^n$ has a zero of multiplicity $n$ at $\omega=0$ with the underlying sinusoid being the constant $1=\ee^{\jj 0}$. A single zero contributes a sinusoidal component to the null space, while a multiple one contributes a subspace of modulated polynomials (i.e., the product of the underlying sinusoid and a series of monomials). The presence of such zeros  makes
$\frac{1}{|\widehat L(\bw)|^2}$ singular and therefore spoils the Bochner integrability criterion.
While this tells us that the conditional inverse of a $\V p$-admissible operator cannot be positive definite, it does not compromise the existence of a symmetric operator $\Op A_{\rm LSI}: \Spc H' \to \Spc S'$ whose kernel is also given by \eqref{Eq:LSIKernel}, but with an appropriate domain extension. Formally, this amounts to replacing the classical inverse Fourier transform $\Fourier^{-1}: L_1(\R^d)\to C_0(\R^d)$ by its generalized counterpart $\Fourier^{-1}: \Spc S'(\R^d)\to \Spc S'(\R^d)$ that can accommodate singularities. 

While a $\V p$-admissible LSI operator has many possible conditional inverses (with the corresponding Green's functions differing by a null-space component), it has a canonical inverse, which is the one obtained by inverse Fourier transformation. In particular, the Green's function $\rho$ and shift-invariant kernel specified by \eqref{Eq:LSIKernel} is unique and necessarily symmetric. Some representative examples for $\Spc S=\Spc S(\R)$ are shown in Table \ref{Tab:Derivatives}. 
Again, under the condition that $\nu_m=\delta_{\V x_m} \in \Spc H'_{\Lop}$ with $\V x_m \in \R^d$, which is equivalent to $\Spc H_\Lop$ being an RKHS, we can then make a direct connection with spline approximation in $\R^d$ by noting that
\eqref{Eq:f0lin2b} then takes the form
\begin{align}
f_\lambda(\V x)=\sum_{m=1}^M a_m \rho(\V x-\V x_m) +
 \sum_{n=1}^{N_0} b_n p_n(\V x),
\end{align}
where $\Spc N_\V p$ is a space of (possibly modulated) polynomials. The resulting function is a $\Lop^\ast\Lop$-spline
with knots (or break points) at the $\V x_m$.
In particular, by plugging in the explicit formulas and parameters for $\Lop=\Dop^N$ in Table \ref{Tab:Derivatives}, we find that 
\begin{align}
f_\lambda(x)=\sum_{m=1}^M a_m |x-x_m|^{2N-1} +
 \sum_{n=1}^{N} b_n x^{n-1},
\end{align}
which is the classic formula of a polynomial spline of degree $(2N-1)$. The same can also be done for the $\V p$-admissible operator $\Op D^{\gamma}$ (fractional derivative of order $\gamma>\tfrac{1}{2}$, so that $N_0=\lfloor \gamma+\tfrac{1}{2} \rfloor\ge 1$), which yields fractional splines \cite{Unser2000a} in the regime $\gamma \notin \N$ where the operator is non-regular.  

\subsection{Exact Discretization}
\label{Sec:Discretization}
The key outcome of Theorem \ref{Theo:GeneralRepSemiHilbert} is that the solutions of a broad class of infinite-dimensional optimization problems in Hilbert spaces lie in a finite-dimensional subspace whose structure depends on the problem formulation, but not on the data. This remarkable property allows us to reformulate the original problem as an equivalent finite-dimensional minimization problem (an exact discretization), which can then be solved using standard numerical optimization techniques.

Because the underlying parametrization is linear, the resulting formulation also suggests that, in the case of an exact fit or a least-squares criterion, this class of problems admits a closed-form solution, irrespective of the particular Hilbert space $\Spc H_\Lop$ in which the solution is defined.
\subsubsection{Tikhonov Regularization and Ridge Regression}
\label{Sec:Tikhonov}
It is instructive to start with the simplest instance of the framework with $N_0=0$ (trivial null
space) and the least-squares loss
\begin{align}
E_{\rm LS}\big(\M y, \V \nu(f)\big)=\sum_{m=1}^M |y_m -\langle \nu_m, f\rangle|^2,
\label{Eq:RidgeLS}
\end{align}
which corresponds to the classical setting of Tikhonov regularization in inverse problems, also
known as ridge regression in statistics. Since $\Spc N_{\V p}=\{0\}$, the matrix $\M P$ is void and
the two parametrizations of Theorem \ref{Theo:GeneralRepSemiHilbert} collapse into the single
expansion $f_\lambda=\sum_{m=1}^M a_m \psi_m$ with $\psi_m=\Op A\{\nu_m\}$. The system matrix
$\M G \in \R^{M \times M}$ with $[\M G]_{m,n}=\langle \nu_m,\psi_n \rangle$ is then the Gram matrix
of the basis $\{\psi_m\}$, because the duality relation \eqref{Eq:RieszPairing} with $\nu=\nu_m$
and $f=\psi_n=\nu^\ast_n$ gives
$\langle \nu_m, \psi_n\rangle=\langle \nu^\ast_m, \psi_n\rangle_{\Spc H_\Lop}=\langle \psi_m, \psi_n\rangle_{\Spc H_\Lop}$.
In particular, $\M G$ is symmetric and positive-definite, the latter being inherited from the linear
independence of the $\nu_m$ through the Riesz pairing. The general treatment of
Section \ref{Sec:InvLS} then reduces to
\begin{align}
\label{Eq:TikAbstract}
f_\lambda= \sum_{m=1}^M a_m \psi_m\quad  \mbox{ with }\quad \M a=(\M G + \lambda \M I)^{-1} \M y,
\end{align}
which is the familiar ridge-regression estimator. The general case that we consider next differs
from this one in a single, but consequential, respect: the presence of $\Spc N_{\V p}$ brings in the
orthogonality constraint of \eqref{Eq:f0lin2b}.
\subsubsection{General Scenario with Non-Trivial Null Space}
Under the assumption that the $\Lop$-compatibility condition in Definition \ref{Def:AdmisPoints} is met, we know from Theorem \ref{Theo:GeneralRepSemiHilbert} that the solution lives in the $M$-dimensional subspace spanned by the basis functions
$\{\widetilde{ \psi}_m\}_{m=1}^{M-N_0} \cup \{p_n\}_{n=1}^{N_0}$. 
Accordingly, we discretize the reconstruction problem
by plugging \eqref{Eq:f0lin2} into the cost functional. Since the measurement model is linear, it suffices to determine the system matrix $\widetilde{\M  G} \in \R^{(M-N_0) \times (M-N_0)}$ for the primary basis functions $\widetilde{ \psi}_m=\Op A_{\widetilde{\V \phi}}\{\widetilde\nu_m\}$ with $\widetilde\nu_m \in \Spc H_\Op A$. Specifically, we have that
\begin{align}
\label{Eq:SysMatTilde}
[\widetilde{\M G}]_{m,n}%
:=\langle\widetilde{\nu}_m,\widetilde{\psi}_n \rangle=\langle\widetilde{\nu}_m,\Op A_{\widetilde{\V \phi}} \{\widetilde{\nu}_n\} \rangle, \quad \widetilde{\V \nu}=(\widetilde \nu_m)=\widetilde{\M Q}^\Tr \V \nu,
\end{align}
with the remarkable property that the matrix $\widetilde{\M G}$ is guaranteed to be positive-definite.
Let us note that $\Op A_{\widetilde{\V \phi}}$ may be substituted by $\Op A$ in \eqref{Eq:SysMatTilde} without affecting $\widetilde{\M G}$, because the two differ by an element of $\Spc N_{\V p}$ and $\widetilde{\V \nu}(\V p)=\M 0$. The positive-definiteness comes as follows. Since $\widetilde{\nu}_m,\widetilde \nu_n \in \Spc H_\Op A=(\Spc H_{\Lop,\widetilde{\V \phi}})'$, the operator $\Op A_{\widetilde{\V \phi}}$ restricted to $\Spc H_\Op A$ is the Riesz map $\Spc H_\Op A \to \Spc H_{\Lop,\widetilde{\V \phi}}$ whose inverse is the normal operator
$(\Lop^\ast\Lop): \Spc H_{\Lop,\widetilde{\V \phi}} \to \Spc H_{\Op A}$ (see Theorem \ref{Theo:NativeSpace}, Item \ref{Item:RieszMaps}). This enables us to recognize the entries of $\widetilde{\M G}$
as inner products in $\Spc H_{\Op A}$ (resp., in $\Spc H_{\Lop,\widetilde{\V \phi}}$) with
\begin{align}
\label{Eq:SysMatTilde2}
[\widetilde{\M G}]_{m,n}&=\langle \widetilde{\nu}_m, \widetilde{\nu}_n\rangle_{\Spc H_\Op A}=\langle \widetilde{\nu}_m,\Op A\{\widetilde{\nu}_n\} \rangle \nonumber
\\&=\langle \widetilde{\psi}_m, \widetilde{\psi}_n\rangle_{\Spc H_{\Lop,\widetilde{\V \phi}}}=\langle \Lop \widetilde{\psi}_m, \Lop \widetilde{\psi}_n\rangle.
\end{align}
In addition, the right-hand side of \eqref{Eq:SysMatTilde2} implies that $\widetilde{\M G}$ coincides with the Gram matrix of
the reduced basis $\widetilde{\V \psi}=(\widetilde{ \psi}_m)$. For
\begin{align*}
f=\sum_{m=1}^{M-N_0} \tilde a_m \widetilde{ \psi}_m + \sum_{n=1}^{N_0} b_n p_n,
\end{align*}
this then yields
\begin{align*}
\langle (\Lop^\ast\Lop)\{f\}, f\rangle =\|\Lop f\|^2_{L_2}=\widetilde{\M a}^\Tr \widetilde{\M  G}\widetilde{\M a},
\end{align*}
which discretizes the regularization term exactly. Finally, by decomposing the measurement operator as $\V \nu
= {\M Q} \widetilde{\V \nu} + {\M P} \widetilde{\V \phi}$, we recast the reconstruction problem as a finite-dimensional optimization with
\begin{align}
\label{Eq:NumOpt1}
\min_{(\widetilde{\M a}, \M b) \in \R^{M-N_0} \times \R^{N_0}} E(\M y, \M Q \widetilde{\M G} \widetilde{\M a}+ \M P \M b) + \lambda \widetilde{\M a}^\Tr \widetilde{\M  G}\widetilde{\M a}.
\end{align}
This kind of problem can then be solved
numerically using some basic iterative algorithm (e.g., steepest descent), or in closed form when the data term is quadratic, as in Section \ref{Sec:Tikhonov}.

Based on the right-hand side of \eqref{Eq:SysMatTilde} and by defining the augmented system matrix $\M G \in \R^{M \times M}$ with
$[\M G]_{m,n}=\langle \nu_m,\psi_{n}\rangle=\langle \nu_m,\Op A\{\nu_n\}\rangle$, we also have
\begin{align}
\label{Eq:SysMatTilde3}
\widetilde{\M  G}=\widetilde{\M Q}^\Tr \M G \widetilde{\M Q}.
\end{align}
The latter is consistent with the relation 
$\widetilde{\V \psi}=\widetilde{\M Q}^\Tr \V \psi$,
which relates
the basis functions of the reduced and over-parametrized representations. Likewise, by mapping the reduced coefficients back to the over-parameterized ones, we
find that $\M Q \widetilde{\M G}=\M G \widetilde{\M Q}$, which, together with \eqref{Eq:SysMatTilde3}, enables the reformulation of \eqref{Eq:NumOpt1} in terms of 
 the augmented system matrix $\M G$.

In order to be complete, we also present the overparameterized discretization
\begin{align}
\label{Eq:OverparamNum}
\min_{(\M a, \M b) \in \R^M \times \R^{N_0}} E(\M y, \M G \M a+ \M P \M b) + \lambda \M a^\Tr  \M G \M a  \quad  \mbox{ s.t. }\quad
\M P^\Tr \M a=\M 0,
\end{align}
which goes hand in hand with the reconstruction formula \eqref{Eq:f0lin2b}.
It is worth noting that, unlike $\widetilde{\M G}$, the augmented system matrix $\M G$ is generally
not positive definite, unless the null space of $\Lop$ is trivial, as in
Section~\ref{Sec:Tikhonov}. It is, however, always symmetric, because $\Op A$ is, and it is
positive-definite on the subspace $\{\M a \in \R^M: \M P^\Tr\M a=\M 0\}$, on which
$\M a^\Tr \M G \M a=\|\Lop f\|^2_{L_2}$ by \eqref{Eq:SysMatTilde3}. Let us also observe that,
while the entries of $\M G$ do depend on which conditional inverse is selected, the quadratic form
$\M a^\Tr \M G \M a$ does not, precisely because $\M P^\Tr\M a=\M 0$. This is a second reason why
the orthogonality constraint is structural rather than cosmetic.

The minimization of \eqref{Eq:OverparamNum} is more delicate than that of \eqref{Eq:NumOpt1} because of the additional orthogonality constraint. This constraint is essential to the formulation since it guarantees the equivalence between the continuous-domain and discrete optimization problems. The subtle point is that the identity $\|\Lop f\|^2_{L_2}=\M a^\Tr \M G \M a$ holds only when the orthogonality condition is satisfied\footnote{The implication of \eqref{Eq:Afprop} in Proposition~\ref{Prop:RangeA} is that the true continuous-domain regularization cost $\|\Lop f\|_{L_2}$ becomes unbounded whenever this condition is violated.}. Consequently, the constraint must be enforced systematically throughout the optimization process.
In particular, this rules out the use of basic steepest-descent methods, in contrast with the unconstrained minimization of \eqref{Eq:NumOpt1}. From a purely numerical standpoint, it is also this constraint that ensures uniqueness of the solution in an otherwise overparameterized setting.

\subsubsection{Least-Squares Approximation with Generalized Splines}
\label{Sec:InvLS}
As an illustration of the exact discretization pipeline, we again consider the least-squares criterion \eqref{Eq:RidgeLS}, but now in the more challenging spline setting where the null space of $\Op L$ is non-trivial. 
The resulting estimator is a generalized smoothing spline which, as we now show, admits a closed-form expression in either parametrization.

We begin with the unconstrained formulation \eqref{Eq:NumOpt1}, whose resolution is
straightforward. Since the cost is quadratic, its minimizer is obtained by computing the gradient
with respect to the unknown coefficients, setting it to zero, and solving the resulting linear
system. Upon ``division'' of the first block row by $\widetilde{\M G}$, which is invertible, this yields
\begin{align}
\label{Eq:LSreduced}
\left(
\begin{array}{c c}
\M Q^\Tr \M Q\widetilde{\M G} + \lambda \M I \ \ &  \M Q^\Tr\M P  \\
\M P^\Tr \M Q\widetilde{\M G}  & \M P^\Tr\M P
\end{array}
\right)
\left(
\begin{array}{c}
  \widetilde{\M a}   \\
  \M b
\end{array}
\right)
=\left(
\begin{array}{c}
\M Q^\Tr    \\
  \M P^\Tr
\end{array}
\right) \M y,
\end{align}
which is a well-posed linear system in $\R^M$.

We can do better by working with the canonical decomposition \eqref{Eq:CanonicalBasis} of
Appendix \ref{App:MeasurementSpace}, whose matrices satisfy $\M Q^\Tr\M P=\M 0$ and
$\M Q^\Tr\M Q=\M I_{M-N_0}$. The system \eqref{Eq:LSreduced} then decouples into
\begin{align}
\label{Eq:LSdecoupled}
\widetilde{\M a}=(\widetilde{\M G} + \lambda \M I)^{-1}\M Q^\Tr\M y,
\qquad
\M b=(\M P^\Tr\M P)^{-1}\M P^\Tr\M y.
\end{align}
The interpretation is transparent: The null-space coefficients $\M b$ are the least-squares fit
\eqref{Eq:LSpol} of the data in $\Spc N_{\V p}$, and they do not depend on $\lambda$. The spline
coefficients $\widetilde{\M a}$ are then obtained by a ridge regression on the residual measurements
$\M Q^\Tr\M y$, in exact analogy with \eqref{Eq:TikAbstract}. In particular,
$\widetilde{\M a}\to \M 0$ as $\lambda \to \infty$, which recovers the linear regression
$f_\infty \in \Spc N_{\V p}$ that was announced after
Theorem \ref{Theo:GeneralRepSemiHilbert}.

In the overparametrized setup, the optimal solution is specified as
\begin{align}
J_{\rm LS}(\M a, \M b| \M y,\lambda)&= \|\M G\M a+\M P\M b-\M y\|_2^2 +  \lambda\, \M a^\Tr \M G \M a\\
\arg \min_{(\M a,\M b)\in \R^M\times \R^{N_0}} & J_{\rm LS}(\M a, \M b| \M y,\lambda)\quad \mbox{ s.t. } \quad\M P^\Tr\M a=\M 0. \label{Eq:OverParProblem}
\end{align}
Here, we proceed indirectly by showing that the Karush-Kuhn-Tucker (KKT) conditions of optimality for \eqref{Eq:OverParProblem} are satisfied whenever
\begin{align}
\begin{cases}\label{Eq:opt1}
(\M G + \lambda \M I) \M a + \M P\M b = \M y\\
\M P^\Tr \M a=\M 0,
\end{cases}
\end{align}
because this configuration both enforces the constraint and annihilates the gradient of $J_{\rm LS}(\M a, \M b| \M y,\lambda)$. 
Since $\M G$ is symmetric, the partial derivatives of $J_{\rm LS}(\M a, \M b| \M y,\lambda)$ 
with respect to $\M a$ and $\M b$ are
\begin{align}
\frac{\partial J_{\rm LS}(\M a, \M b| \M y,\lambda)}{\partial \M a}&=2 \M G(\M G\M a+\M P\M b-\M y)+2 \lambda \M G \M a,\nonumber\\
&=2 \M G\Big((\M G + \lambda \M I) \M a + \M P\M b - \M y\Big),\label{Eq:Grada}\\
\frac{\partial J_{\rm LS}(\M a, \M b| \M y,\lambda)}{\partial \M b}&=2 \M P^\Tr(\M G\M a+\M P\M b-\M y).
\label{Eq:Gradb}
\end{align}
Substituting the expression of $\M y$ given by \eqref{Eq:opt1} into
\eqref{Eq:Grada} and \eqref{Eq:Gradb} immediately shows that both partial derivatives vanish, thereby proving the claim.
The remarkable aspect of \eqref{Eq:opt1} is that the associated KKT Lagrange multiplier vanishes.
This is not a coincidence. Eq.~\eqref{Eq:Grada} shows that $\partial J_{\rm LS}/\partial \M a$
factors through $\M G$, so that any solution of the first line of \eqref{Eq:opt1} annihilates that
gradient by itself, without the help of a multiplier, while the constraint is carried by the second
line. It follows that the minimal value of $J_{\rm LS}$ in \eqref{Eq:OverParProblem} coincides with
that of the corresponding unconstrained problem. The difference lies in the multiplicity of
solutions: as soon as $N_0>0$, the pair $(\M a,\M b)$ is overparameterized---there are $M+N_0$
unknowns for $M$ measurements---so that the unconstrained problem admits infinitely many
minimizers. Among them, only one satisfies the orthogonality condition $\M P^\Tr \M a=\M 0$, which
is the solution of \eqref{Eq:OverParProblem}. The system \eqref{Eq:opt1} is indeed invertible,
because $\M P$ is of full column rank and because $(\M G + \lambda \M I)$ is positive-definite on
$\{\M a: \M P^\Tr\M a=\M 0\}$.

The optimal spline is found by solving \eqref{Eq:opt1},  which results in
the closed-form solution
\begin{align*}
\left(
\begin{array}{c}
  \M a   \\
  \M b
\end{array}
\right)=\left(
\begin{array}{c c}
(\M G + \lambda \M I) &  \M P  \\
\M P^\Tr & \M 0
\end{array}
\right)^{-1}
\left(
\begin{array}{c}
 \M y   \\
  \M 0
\end{array}
\right).
\end{align*}
This solution remains valid as $\lambda\to 0$ because the orthogonality property is required in all cases to ensure that the solution is part of the native space.
In the limit, we get the generalized interpolant $f_{0} \in \Spc H_\Lop$ with expansion coefficients $(\M a_{(0)},\M b_{(0)})$, which is such that $\V \nu(f_{0})%
=\M y$.
Based on \eqref{Eq:opt1}, it is then possible to simplify the corresponding spline energy as
\begin{align*}
\|\Lop f_{0}\|^2_{L_2}&=\M a_{(0)}^\Tr \M G \M a_{(0)}= \M a_{(0)}^\Tr  (\M y- \M P \M b_{(0)})= \M a_{(0)}^\Tr  \M y- \M b^\Tr_{(0)}(\M P^\Tr\M a_{(0)})
=\M a_{(0)}^\Tr \M y.
\end{align*}
At the other extreme for $\lambda\to \infty$, we have that
\begin{align*}
\left(
\begin{array}{c}
  \M a_{(\infty)}   \\
  \M b_{(\infty)}
\end{array}
\right)=\lim_{\lambda\to \infty}\left(
\begin{array}{c c}
\lambda \M I &  \M P  \\
\M P^\Tr & \M 0
\end{array}
\right)^{-1}
\left(
\begin{array}{c}
 \M y   \\
  \M 0
\end{array}
\right).
\end{align*}
This yields $ \M a_{(\infty)}=\M 0$ and
\begin{align*}
\M b_{(\infty)}=(\M P^\Tr  \M P)^{-1} \M P^\Tr   \M y,
\end{align*}
where the latter represents the orthogonal projection of $\M y$ in the space spanned by the column
vectors of $\M P$, in agreement with the second formula of \eqref{Eq:LSdecoupled}.
\section{Conclusion}
\label{Sec:Conclusion}
The proposed formulation with its representer theorem
(Theorem \ref{Theo:GeneralRepSemiHilbert}) unifies interpolation, smoothing splines, ridge
regression, Tikhonov regularization, and kernel estimation as special cases of a single
statement over the native space of any $\V p$-admissible operator. The underlying hypotheses delineate the
class---acceptable operators, maximal search space, admissible measurements, and losses---for
which the solutions of these infinite-dimensional problems are confined to a finite-dimensional
subspace, while the exact discretization of Section \ref{Sec:Discretization} provides the common
resolution pipeline, with closed-form solutions when the data-fidelity term is quadratic.

Having the complete class under the same umbrella, in a self-contained manner, brings a
twofold benefit. On the conceptual side, a single mechanism---the interplay between the Riesz
map of the native space and the conditioned basis---explains the parametric form of all these
estimators at once. On the practical side, one implementation serves the whole class: the
concrete work of the practitioner reduces to the determination of the basis functions of
\eqref{Eq:f0lin2} and of the Gram matrix specified by \eqref{Eq:SysMatTilde}, while the rest is
generic linear algebra.
The conditioned parametrization, finally, is the bridge to the stochastic side of the story:
Part III \cite{Unser2026c} identifies the generalized splines of this paper as the
minimum-mean-square-error estimators of generalized Gaussian processes, which extends the
classical correspondence of Kimeldorf and Wahba \cite{Kimeldorf1970} to the present level of
generality.

\appendix
\section{Abstract Representer Theorem}
\label{App:RepTheorem}
In the interest of self-containedness, we prove the following classical result, known as the representer theorem: The solutions to minimum-norm problems in Hilbert spaces subject to finitely many linear constraints lie in the finite-dimensional subspace spanned by the Riesz representers of the constraint functionals.

\begin{proof}[Proof of Theorem  \ref{Theo:MinNorm}]
Every $\nu_m \in \Spc H'$ has a unique Riesz representer
$\nu_m^\ast=\Op J_{\Spc H'}\{\nu_m\} \in \Spc H$, which satisfies
$\langle \nu_{m'}, \nu^\ast_m\rangle=\langle \nu_{m'}^\ast,\nu_m^\ast \rangle_{\Spc H}
=\langle \nu_{m'},\nu_m \rangle_{\Spc H'}=[\M G]_{m',m}$, where $\M G \in \R^{M \times M}$ is
the {\em Gram matrix} of the system. The assumption of linear independence ensures that $\M G$
is positive-definite and hence invertible.

\medskip
\emph{Lagrangian and duality}:
By introducing the multipliers $\V\lambda=(\lambda_m) \in \R^M$, we write the Lagrangian of
\eqref{Eq:Ginterpol} as
\begin{align*}
\mathcal L(f,\V\lambda)
&= \frac12 \|f\|_{\Spc H}^2
+ \sum_{m=1}^M \lambda_m \big( z_m - \langle \nu_m , f \rangle \big)\\
&= \frac12
   \Big\| f - \sum_{m=1}^M \lambda_m \nu_m^\ast \Big\|_{\Spc H}^2
   - \frac12
   \Big\| \sum_{m=1}^M \lambda_m \nu_m^\ast \Big\|_{\Spc H}^2
   + \V\lambda^\Tr \M z,
\end{align*}
where we have used the defining property
$\langle \nu_m, f\rangle=\langle \nu^\ast_m, f\rangle_{\Spc H}$ of the Riesz representers and
completed the square. The infimum over $f \in \Spc H$ is attained at
$f=\sum_{m=1}^M \lambda_m \nu_m^\ast$, which identifies the dual function as
\begin{align*}
\inf_{f\in\Spc H} \mathcal L(f,\V\lambda)
= \V\lambda^\Tr \M z
  - \frac12
    \Big\| \sum_{m=1}^M \lambda_m \nu_m \Big\|_{\Spc H'}^2
=\V\lambda^\Tr \M z-\tfrac{1}{2}\V \lambda^\Tr \M G \V\lambda .
\end{align*}
This is the criterion of \eqref{Eq:DualProb}, which proves the claim of duality. The quadratic
form is strictly concave, and its (unique) maximum is achieved for $\M a=\M G^{-1} \M z$.

\medskip
\emph{Resolution of the primal problem \eqref{Eq:Ginterpol}}:
Let $g_\M a=\sum_{m=1}^M a_m \nu_m^\ast$ with $\M a=\M G^{-1} \M z$. Since
$\langle \nu_{m'},g_\M a\rangle=\sum_{m=1}^M a_m [\M G]_{m',m}=z_{m'}$ for all $m'$, the
feasibility set $\Spc A_\V z \eqdef \{ f \in \Spc H : \V \nu(f)=\M z \}$ is non-empty. It is
the affine set $\Spc A_\V z=g_\M a + \Spc H_0^\perp$, where
$\Spc H_0={\rm span}\{\nu^\ast_m\}_{m=1}^M$ and
$\Spc H_0^\perp=\{f\in \Spc H: \langle \nu_m, f\rangle=0,\ m=1,\dots,M\}$. Any feasible
$f=g_\M a+h$ with $h \in \Spc H_0^\perp$ then satisfies the Pythagorean relation
\begin{align*}
\|f\|^2_{\Spc H}= \|g_\M a\|^2_{\Spc H} + \|f-g_\M a\|^2_{\Spc H},
\end{align*}
which proves that $f_0=g_\M a$ is the unique minimum-norm solution, with
$\|f_0\|^2_{\Spc H}=\M a^\Tr \M G \M a=\M z^\Tr \M G^{-1} \M z$.
\end{proof}
\section{Decomposition of the Measurement Space}
\label{App:MeasurementSpace}
We first prove Lemma \ref{Prop:DirectSum} and then present a practical solution to the design of a measurement-adapted decomposition.
\begin{proof}[Proof of Lemma \ref{Prop:DirectSum}]
Since $\M P$ is of full column rank $N_0$, there necessarily exists some complementary matrix $\M Q\in \R^{M \times (M-N_0)}$
such that  $[\M Q \ \ \M P] \in \R^{M \times M}$ is invertible. We then set
$\left[
\begin{array}{c}
 \widetilde{\M Q}^\Tr\\
    \widetilde{\M P}^\Tr
 \end{array}
\right]=[\M Q \ \ \M P]^{-1}$ and expand \eqref{Eq:Biortho2} to get\begin{align}
\left[
\begin{array}{cc}
 \widetilde{\M Q}^\Tr\M Q & \widetilde{\M Q}^\Tr\M P \\
    \widetilde{\M P}^\Tr\M Q &  \widetilde{\M P}^\Tr\M P
 \end{array}
\right] = \left[\begin{array}{cc}
\M I_{M-N_0} & \M 0 \\
\M 0 &  \M I_{N_0}
 \end{array}
\right].
\label{Eq:InverseBiortho}
\end{align}
Next, we write \eqref{Eq:ReducedBasis} and \eqref{Eq:ReducedBasis2} in matrix form as
\begin{align}
\label{Eq:directsumRM}
\left[
\begin{array}{c}
\widetilde{\V \nu}  \\
\widetilde{\V \phi}  
 \end{array}
\right] = \left[
\begin{array}{c}
 \widetilde{\M Q}^\Tr\\
    \widetilde{\M P}^\Tr
 \end{array}
\right] \V \nu.
\end{align}
This is a linear system that can be inverted to yield
\begin{align}
\V \nu
= \left[
\begin{array}{c}
 \widetilde{\M Q}^\Tr\\
    \widetilde{\M P}^\Tr
 \end{array}
\right]^{-1} \left[
\begin{array}{c}
\widetilde{\V \nu}  \\
\widetilde{\V \phi}  
 \end{array}
\right] =
 \left[
\begin{array}{cc}
 {\M Q}  &  {\M P} 
 \end{array}
\right] \left[
\begin{array}{c}
\widetilde{\V \nu}  \\
\widetilde{\V \phi}  
 \end{array}
\right],
\end{align}
which proves  \eqref{Eq:Decompnu}.
Based on these definitions and on \eqref{Eq:InverseBiortho}, we get $\widetilde{\V \nu}(\V p)=\widetilde{\M Q}^\Tr\V \nu(\V p)=\widetilde{\M Q}^\Tr\M P=\M 0$ ($\V p$-orthogonality) and
$\widetilde{\V \phi}(\V p)=\widetilde{\M P}^\Tr\V \nu(\V p)=\widetilde{\M P}^\Tr\M P=\M I_{N_0}$ (biorthogonality). Likewise, we have that
$ \widetilde{\M Q}^\Tr\M Q=\M Q^\Tr \widetilde{\M Q}=\M I_{M-N_0}$ and 
$ \M P^\Tr \widetilde{\M P}=(\widetilde{\M P}^\Tr\M P)^\Tr=\M I_{N_0}$.

\end{proof}

Our proposed change of basis
 is to take
\begin{align}
\label{Eq:CanonicalBasis}
\widetilde{\M P}=\M P(\M P^\Tr\M P)^{-1},
\end{align}
which satisfies the biorthonormality condition $\widetilde{\M P}^\Tr \M P = \M I_{N_0}$ and has
the distinguishing property that ${\rm range}(\widetilde{\M P})={\rm range}(\M P)$.
For $\M Q$,
we take any matrix whose columns form an orthonormal basis of the null space of $\M P^\Tr$
(e.g., computed from the singular value or QR decomposition of $\M P$). Then,
$\M Q^\Tr\M P=\M 0$ and $\M Q^\Tr\M Q=\M I_{M-N_0}$. Moreover, we also have that
$\widetilde{\M P}^\Tr\M Q=(\M P^\Tr\M P)^{-1}\M P^\Tr\M Q=\M 0$.
Hence, the four conditions of \eqref{Eq:InverseBiortho} are met
with $\widetilde{\M Q}=\M Q$, so that no correction of $\M Q$ is required.

This choice is canonical in the sense that it completely decouples the two components of
the solution (see Section \ref{Sec:Discretization}):  Since
$\M P\widetilde{\M P}^\Tr=\M P(\M P^\Tr\M P)^{-1}\M P^\Tr$ and
$\M Q\M Q^\Tr=\M I_M-\M P(\M P^\Tr\M P)^{-1}\M P^\Tr$ are the orthogonal projectors onto
${\rm range}(\M P)$ and its orthogonal complement, respectively, it realizes the orthogonal
splitting of the measurement space as \begin{align}
\R^M
=
{\rm range}(\M P)\oplus{\rm range}(\M P)^\perp. 
\end{align}
In this configuration, the
component $\widetilde{\V \phi}=\widetilde{\M P}^\Tr\V \nu$ of \eqref{Eq:ReducedBasis2} carries
the least-squares fit of the data onto $\Spc N_\V p$, while $\widetilde{\V \nu}=\M Q^\Tr\V \nu$
carries the residual measurements. This is precisely the decoupling exploited in
Section \ref{Sec:InvLS}.

\subsection*{Acknowledgment}
The author used Claude (Anthropic) as an interactive assistant during the preparation of
this paper, for proofreading, for checking and simplifying proofs. All results were verified by the author,
who takes full responsibility for the content.
This research was supported by the Swiss National Science Foundation (SNSF) under Grant 200020\_219356.

\bibliographystyle{unsrtinitials}
\bibliography{SplineOps}

@Article{deBoor1966,
  author  = {de Boor, Carl and Lynch, Robert E.},
  title   = {On Splines and Their Minimum Properties},
  journal = {Journal of Mathematics and Mechanics},
  year    = {1966},
  volume  = {15},
  number  = {6},
  pages   = {953--969},
}

@Article{Jerome1969,
  author  = {Jerome, Joseph W. and Schumaker, Larry L.},
  title   = {On {L}g-splines},
  journal = {Journal of Approximation Theory},
  year    = {1969},
  volume  = {2},
  number  = {1},
  pages   = {29--49},
  doi     = {10.1016/0021-9045(69)90029-X},
}

@Book{Prenter:1975,
  author    = {Prenter, Paddy M.},
  title     = {Splines and Variational Methods},
  publisher = {Wiley},
  address   = {New York},
  year      = {1975},
}

@Book{Schumaker2007,
  author    = {Schumaker, Larry L.},
  title     = {Spline Functions: {B}asic Theory},
  publisher = {Cambridge University Press},
  address   = {Cambridge},
  year      = {2007},
  edition   = {3rd},
  series    = {Cambridge Mathematical Library},
  doi       = {10.1017/CBO9780511618994},
}

@Article{Schoenberg1964,
  author  = {Schoenberg, Isaac J.},
  title   = {Spline functions and the problem of graduation},
  journal = {Proceedings of the National Academy of Sciences of the USA},
  year    = {1964},
  volume  = {52},
  number  = {4},
  pages   = {947--950},
  doi     = {10.1073/pnas.52.4.947},
}

@InCollection{Duchon1977,
  author    = {Duchon, Jean},
  title     = {Splines minimizing rotation-invariant semi-norms in {S}obolev spaces},
  booktitle = {Constructive Theory of Functions of Several Variables},
  editor    = {Schempp, Walter and Zeller, Karl},
  series    = {Lecture Notes in Mathematics},
  volume    = {571},
  pages     = {85--100},
  publisher = {Springer},
  address   = {Berlin},
  year      = {1977},
  doi       = {10.1007/BFb0086566},
}

@InCollection{Schaback1999,
  author    = {Schaback, Robert},
  title     = {Native {H}ilbert spaces for radial basis functions {I}},
  booktitle = {New Developments in Approximation Theory},
  editor    = {Buhmann, Martin D. and Mache, Detlef H. and Felten, Michael
               and M{\"u}ller, Manfred W.},
  series    = {International Series of Numerical Mathematics},
  volume    = {132},
  pages     = {255--282},
  publisher = {Birkh{\"a}user},
  address   = {Basel},
  year      = {1999},
  doi       = {10.1007/978-3-0348-8696-3_16},
}

@Article{Schaback2000,
  author  = {Schaback, Robert},
  title   = {A unified theory of radial basis functions: {N}ative {H}ilbert
             spaces for radial basis functions {II}},
  journal = {Journal of Computational and Applied Mathematics},
  year    = {2000},
  volume  = {121},
  number  = {1--2},
  pages   = {165--177},
  doi     = {10.1016/S0377-0427(00)00345-9},
}

@Article{Aronszajn1950,
  author  = {Aronszajn, Nachman},
  title   = {Theory of reproducing kernels},
  journal = {Transactions of the American Mathematical Society},
  year    = {1950},
  volume  = {68},
  number  = {3},
  pages   = {337--404},
  doi     = {10.1090/S0002-9947-1950-0051437-7},
}

@Article{Micchelli1986,
  author  = {Micchelli, Charles A.},
  title   = {Interpolation of scattered data: {D}istance matrices and
             conditionally positive definite functions},
  journal = {Constructive Approximation},
  year    = {1986},
  volume  = {2},
  number  = {1},
  pages   = {11--22},
  doi     = {10.1007/BF01893414},
}

@Book{Wendland2005,
  author    = {Wendland, Holger},
  title     = {Scattered Data Approximation},
  series    = {Cambridge Monographs on Applied and Computational Mathematics},
  volume    = {17},
  publisher = {Cambridge University Press},
  address   = {Cambridge},
  year      = {2005},
  doi       = {10.1017/CBO9780511617539},
}

@Article{Kanagawa2018,
  author  = {Kanagawa, Motonobu and Hennig, Philipp and
             Sejdinovic, Dino and Sriperumbudur, Bharath K.},
  title   = {{G}aussian processes and kernel methods: {A} review on
             connections and equivalences},
  journal = {arXiv preprint arXiv:1807.02582},
  year    = {2018},
  doi     = {10.48550/arXiv.1807.02582},
}

@Article{Grothendieck1955,
  author  = {Grothendieck, Alexander},
  title   = {Produits tensoriels topologiques et espaces nucl{\'e}aires},
  journal = {Memoirs of the American Mathematical Society},
  year    = {1955},
  volume  = {16},
  pages   = {1--140},
  doi     = {10.1090/memo/0016},
}

@Book{Pietsch1972,
  author    = {Pietsch, Albrecht},
  title     = {Nuclear Locally Convex Spaces},
  series    = {Ergebnisse der Mathematik und ihrer Grenzgebiete},
  volume    = {66},
  publisher = {Springer},
  address   = {Berlin},
  year      = {1972},
  pages     = {ix+193},
}

@Book{Treves2006,
  author    = {Tr{\`e}ves, Fran{\c{c}}ois},
  title     = {Topological Vector Spaces, Distributions and Kernels},
  publisher = {Dover Publications},
  address   = {Mineola, NY},
  year      = {2006},
  note      = {Unabridged republication of the 1967 Academic Press edition},
}

@Unpublished{Unser2026a,
  author = {Unser, M.},
  note   = {Preprint, \'Ecole polytechnique f\'ed\'erale de Lausanne (EPFL)},
  title  = {Splines and Operators: {P}art {I}---{A}dmissibility and Native Spaces},
  year   = {2026},
}

@BOOK{Berlinet2004,
  title = {Reproducing Kernel Hilbert Spaces in Probability and Statistics},
  publisher = {Kluwer Academic Boston},
  year = {2004},
  author = {Berlinet, Alain and Thomas-Agnan, Christine},
  volume = {3},
}

@BOOK{Ekeland1999,
  title = {Convex Analysis and Variational Problems},
  publisher = {{SIAM}},
  year = {1999},
  author = {Ekeland, Ivar and Temam, Roger},
  volume = {28},
  series = {Classics in Applied Mathematics}
}

@Article{Unser2000a,
  author    = {Unser, M. and Blu, T.},
  journal   = {{SIAM} Review},
  title     = {Fractional Splines and Wavelets},
  year      = {2000},
  month     = mar,
  number    = {1},
  pages     = {43--67},
  volume    = {42},
}

@Article{Unser2022,
  author    = {Michael Unser and Shayan Aziznejad},
  journal   = {Applied and Computational Harmonic Analysis},
  title     = {Convex optimization in sums of {B}anach spaces},
  year      = {2022},
  month     = jan,
  pages     = {1--25},
  volume    = {56},
  doi       = {10.1016/j.acha.2021.07.002},
  publisher = {Elsevier {BV}},
}

@Book{Wahba1990,
  author    = {Wahba, Grace},
  publisher = {Society for Industrial and Applied Mathematics},
  title     = {Spline Models for Observational Data},
  year      = {1990},
  address   = {Philadelphia},
}

@Unpublished{Unser2026c,
  author = {Unser, M.},
  note   = {Preprint, \'Ecole polytechnique f\'ed\'erale de Lausanne (EPFL)},
  title  = {Variational Splines and Generalized Gaussian Processes},
  year   = {2026},
}

@Article{Kimeldorf1970,
  author    = {Kimeldorf, G. and Wahba, G.},
  journal   = {The Annals of Mathematical Statistics},
  title     = {A correspondence between {B}ayesian estimation on stochastic processes and smoothing by splines},
  year      = {1970},
  number    = {2},
  pages     = {495--502},
  volume    = {41},
}

@Article{Kimeldorf1971,
  author  = {George Kimeldorf and Grace Wahba},
  journal = {Journal of Mathematical Analysis and Applications},
  title   = {Some results on {T}chebycheffian spline functions},
  year    = {1971},
  number  = {1},
  pages   = {82--95},
  volume  = {33},
}

@Article{Parzen1961,
  author    = {Parzen, Emanuel},
  journal   = {The Annals of Mathematical Statistics},
  title     = {An Approach to Time Series Analysis},
  year      = {1961},
  issn      = {0003-4851},
  month     = dec,
  number    = {4},
  pages     = {951--989},
  volume    = {32},
  doi       = {10.1214/aoms/1177704840},
  publisher = {Institute of Mathematical Statistics},
}

@Article{Kailath1970,
  author  = {Kailath, T.},
  journal = {Proceedings of the IEEE},
  title   = {The innovations approach to detection and estimation theory},
  year    = {1970},
  issn    = {0018-9219},
  month   = may,
  number  = {5},
  pages   = {680--695},
  volume  = {58},
  doi     = {10.1109/PROC.1970.7723},
}

@Article{Weinert1974,
  author    = {Weinert, H. L. and Kailath, T.},
  journal   = {Annals of Statistics},
  title     = {Stochastic Interpretations and Recursive Algorithms for Spline Functions},
  year      = {1974},
  number    = {4},
  pages     = {787--794},
  volume    = {2},
}

@INPROCEEDINGS{Scholkopf2001,
  author = {Sch{\"o}lkopf, Bernhard and Herbrich, Ralf and Smola, Alex J.},
  title = {A Generalized Representer Theorem},
  booktitle = {Computational Learning Theory},
  year = {2001},
  editor = {Helmbold, David and Williamson, Bob},
  pages = {416--426},
  publisher = {Springer Berlin Heidelberg}
}

@Article{Hofmann2008,
  author   = {Hofmann, Thomas and Sch\"olkopf, Bernhart and Smola, Alexander J.},
  journal  = {Annals of Statistics},
  title    = {Kernel methods in machine learning},
  year     = {2008},
  number   = {3},
  pages    = {1171--1220},
  volume   = {36},
}

@ARTICLE{Hoerl1970,
  author = {Hoerl, A. E. and Kennard, R. W.},
  title = {Ridge regression: Biased estimation for nonorthogonal problems},
  journal = {Technometrics},
  year = {1970},
  volume = {12},
  pages = {55--67},
  number = {1}
}

@BOOK{Tikhonov1977,
  title = {Solution of Ill-posed Problems},
  publisher = {Winston \& Sons},
  year = {1977},
  author = {Tikhonov, A. N. and Arsenin, V. Y.},
  address = {Washington},
}

@Article{Gupta2018,
  author  = {Gupta, Harshit and Fageot, Julien and Unser, Michael},
  journal = {{IEEE} Transactions on Signal Processing},
  title   = {Continuous-Domain Solutions of Linear Inverse Problems with {T}ikhonov {\textit{versus}} Generalized {TV} Regularization},
  year    = {2018},
  month   = sep,
  number  = {17},
  pages   = {4670--4684},
  volume  = {66},
}

@Article{Badoual2018,
  author  = {Badoual, A. and Fageot, J. and Unser, M.},
  journal = {{IEEE} Transactions on Signal Processing},
  title   = {Periodic Splines and {G}aussian Processes for the Resolution of Linear Inverse Problems},
  year    = {2018},
  month   = nov,
  number  = {22},
  pages   = {6047--6061},
  volume  = {66},
}

@Article{Unser_2020,
  author    = {Michael Unser},
  journal   = {Foundations of Computational Mathematics},
  title     = {A Unifying Representer Theorem for Inverse Problems and Machine Learning},
  year      = {2021},
  month     = sep,
  number    = {4},
  pages     = {941--960},
  volume    = {21},
  doi       = {10.1007/s10208-020-09472-x},
}

@BOOK{Unser2014book,
  title = {An Introduction to Sparse Stochastic Processes},
  publisher = {Cambridge University Press},
  year = {2014},
  author = {Unser, M. and Tafti, P. D.},
}

@Article{Boyer2018,
  author  = {Boyer, Claire and Chambolle, Antonin and De Castro, Yohann and Duval, Vincent and De Gournay, Fr{\'e}d{\'e}ric and Weiss, Pierre},
  journal = {SIAM Journal of Optimization},
  title   = {On Representer Theorems and Convex Regularization},
  year    = {2019},
  number  = {2},
  pages   = {1260--1281},
  volume  = {29},
}

@Article{Parhi2021b,
  author    = {Rahul Parhi and Robert D Nowak},
  journal   = {Journal of Machine Learning Research},
  title     = {Banach Space Representer Theorems for Neural Networks and Ridge Splines},
  year      = {2021},
  number    = {43},
  pages     = {1--40},
  volume    = {22},
}

@Article{Bredies2018,
  author  = {Bredies, Kristian and Carioni, Marcello},
  journal = {Calculus of Variations and Partial Differential Equations},
  title   = {Sparsity of solutions for variational inverse problems with finite-dimensional data},
  year    = {2020},
  number  = {14},
  pages   = {26},
  volume  = {59},
}

@Article{Bartolucci2021,
  author        = {Francesca Bartolucci and Ernesto {De Vito} and Lorenzo Rosasco and Stefano Vigogna},
  journal       = {Applied and Computational Harmonic Analysis},
  title         = {Understanding neural networks with reproducing kernel {B}anach spaces},
  year          = {2023},
  pages         = {194--236},
  volume        = {62},
  archiveprefix = {arXiv},
  eprint        = {2109.09710},
  primaryclass  = {stat.ML},
}

@Article{Anselone1968,
  author  = {Anselone, P. M. and Laurent, P. J.},
  journal = {Numerische Mathematik},
  title   = {A General Method for the Construction of Interpolating or Smoothing Spline-Functions},
  year    = {1968},
  volume  = {12},
  pages   = {66--82},
}

@Book{Atteia1992,
  author    = {Atteia, M.},
  publisher = {North-Holland},
  title     = {Hilbertian Kernels and Spline Functions},
  year      = {1992},
  address   = {Amsterdam},
  series    = {Studies in Computational Mathematics},
  volume    = {4},
}

@Article{Fasshauer2011,
  author  = {Fasshauer, G. E. and Ye, Q.},
  journal = {Numerische Mathematik},
  title   = {Reproducing Kernels of Generalized {S}obolev Spaces via a {G}reen Function Approach with Distributional Operators},
  year    = {2011},
  volume  = {119},
  number  = {3},
  pages   = {585--611},
}

@Article{Scheuerer2013,
  author  = {Scheuerer, M. and Schaback, R. and Schlather, M.},
  journal = {European Journal of Applied Mathematics},
  title   = {Interpolation of Spatial Data---{A} Stochastic or a Deterministic Problem?},
  year    = {2013},
  volume  = {24},
  number  = {4},
  pages   = {601--629},
}

\end{document}